\documentclass[11pt]{amsart}

\usepackage{amsmath, amssymb, amscd, mathrsfs, url, pinlabel,setspace,hyperref,verbatim}
\usepackage[margin=1.25in,marginparwidth=1in,centering,letterpaper,dvips]{geometry}
\usepackage{color,dcpic,latexsym,graphicx,epstopdf,comment}
\usepackage[all]{xy}
\usepackage{tikz,pgfplots}
\usepackage{tikz-cd,circuitikz}
\usepackage{enumitem,upquote,tabularx,textcomp,setspace,color}
\usepackage{makecell}
\usepackage{caption} 
\xyoption{rotate}

\newtheorem*{rep@theorem}{\rep@title}
\newcommand{\newreptheorem}[2]{%
\newenvironment{rep#1}[1]{%
 \def\rep@title{#2 \ref{##1}}%
 \begin{rep@theorem}}%
 {\end{rep@theorem}}}
\makeatother

\newtheorem {theorem}{Theorem}
\newreptheorem{theorem}{Theorem}
\newtheorem {lemma}[theorem]{Lemma}
\newtheorem {proposition}[theorem]{Proposition}
\newtheorem {corollary}[theorem]{Corollary}

\numberwithin{equation}{section}
\numberwithin{theorem}{section}

\theoremstyle{definition}
\newtheorem{definition}[theorem]{Definition}
\newtheorem{construction}[theorem]{Construction}
\newtheorem{data}[theorem]{Data}
\newtheorem{notation}[theorem]{Notation}

\newtheorem{remark}[theorem]{Remark}
\newtheorem{ques}[theorem]{Question}

\newtheorem{example}[theorem]{Example}
\newtheorem*{ack}{Acknowledgement}
\newtheorem*{org}{Organization}

\newlist{pcases}{enumerate}{1}
\setlist[pcases]{
  label=\bf{Case~\arabic*:}\protect\thiscase.~,
  ref=\arabic*,
  align=left,
  labelsep=0pt,
  leftmargin=0pt,
  labelwidth=0pt,
  parsep=0pt
}
\newcommand{\case}[1][]{%
  \if\relax\detokenize{#1}\relax
    \def\thiscase{}%
  \else
    \def\thiscase{~#1}%
  \fi
  \item
}

\newcommand{\bbslash}{\backslash\backslash}
\newcommand{\Z}{\mathbb{Z}}

\newcommand{\rk}{\textrm{rk}\,}

\newcommand{\C}{\mathbb{C}}

\newcommand{\F}{\mathbb{F}}
\newcommand{\Q}{\mathbb{Q}}

\newcommand{\Mod}{\mathrm{Mod}}

\newcommand{\cH}{\mathcal{H}}
\newcommand{\tcH}{\widetilde{\mathcal{H}}}

\newcommand{\cR}{\mathcal{R}}

\newcommand{\bI}{\mathbb{I}}

\newcommand{\bK}{\mathbb{K}}

\newcommand{\fs}{\mathfrak{s}}

\newcommand{\wti}[1]{{\widetilde{#1}}}

\DeclareFontFamily{U}{mathx}{\hyphenchar\font45}
\DeclareFontShape{U}{mathx}{m}{n}{
      <5> <6> <7> <8> <9> <10>
      <10.95> <12> <14.4> <17.28> <20.74> <24.88>
      mathx10
      }{}
\DeclareSymbolFont{mathx}{U}{mathx}{m}{n}
\DeclareFontSubstitution{U}{mathx}{m}{n}
\DeclareMathAccent{\widecheck}{0}{mathx}{"71}

\newcommand{\img}{\operatorname{Im}}
\newcommand{\gr}{{\operatorname{gr}}}
\newcommand{\id}{\operatorname{id}}
\newcommand{\Dc}[2]{D^{#1}_{\mathrm{c},#2}}
\newcommand{\Dcc}[2]{D^{#1}_{\mathrm{cc},#2}}

\newcommand{\pt}{\mathrm{pt}}

\usetikzlibrary{calc,intersections}
\tikzset{every picture/.style=thick}
\tikzset{link/.style = { white, double = black, line width = 1.75pt, double distance = 1.25pt, looseness=1.75 }}
\tikzset{crossing/.style = {draw, circle, dotted, minimum size=0.5cm, inner sep=0, outer sep=0}}
\pgfplotsset{compat=1.12}

\usepackage{extarrows}
\usepackage{mathabx}
\usepackage{amsbsy}
\usepackage{appendix}
\usepackage{float}
\usepackage{subfigure}
\usepackage[numbers,sort&compress]{natbib}
\usepackage{amsthm}
\usepackage{geometry}
\usepackage{fancyhdr}%Heads
\usepackage{overpic}%pictures with text
\usepackage{mathrsfs}
\usepackage{colonequals}
\usepackage{microtype}
\usepackage{diagbox}
\usepackage{tabularx}

\newcommand{\bpf}{\begin{proof}}
\newcommand{\epf}{\end{proof}}
\newcommand{\bthm}{\begin{theorem}}
\newcommand{\ethm}{\end{theorem}}
\newcommand{\bprop}{\begin{proposition}}
\newcommand{\eprop}{\end{proposition}}
\newcommand{\bcor}{\begin{corollary}}
\newcommand{\ecor}{\end{corollary}}
\newcommand{\blem}{\begin{lemma}}
\newcommand{\elem}{\end{lemma}}
\newcommand{\bdefn}{\begin{definition}}
\newcommand{\edefn}{\end{definition}}
\newcommand{\bcons}{\begin{construction}}
\newcommand{\econs}{\end{construction}}
\newcommand{\bdata}{\begin{data}}
\newcommand{\edata}{\end{data}}
\newcommand{\bexmp}{\begin{example}}
\newcommand{\eexmp}{\end{example}}
\newcommand{\brem}{\begin{remark}}
\newcommand{\erem}{\end{remark}}
\newcommand{\bnot}{\begin{notation}}
\newcommand{\enot}{\end{notation}}
\newcommand{\benu}{\begin{enumerate}}
\newcommand{\benum}{\begin{enumerate}[leftmargin=*]}
\newcommand{\eenu}{\end{enumerate}}
\newcommand{\beq}{\begin{equation}}
\newcommand{\eeq}{\end{equation}}

\newcommand{\ep}{\epsilon}

\newcommand{\op}{\oplus}

\newcommand{\aand}{~\mathrm{and}~}

\newcommand{\xra}{\xrightarrow}

\newcommand{\nkp}[1]{\nu^{\sharp,\mathbb{K}}_{+}(#1)}
\newcommand{\nkm}[1]{\nu^{\sharp,\mathbb{K}}_{-}(#1)}
\newcommand{\nk}[1]{\nu^{\sharp}_{\mathbb{K}}(#1)}

\definecolor{lygreen}{HTML}{016646}
\title{Instanton dimensions of knot surgeries over arbitrary fields}

\author{Zhenkun Li}
\address{Academy of Mathematics and Systems science\\Chinese Academy of Science}
\email{zhenkun@amss.ac.cn}

\author{Fan Ye}
\address{Department of Mathematics\\Harvard University}
\email{fanye@math.harvard.edu}

\makeatletter

\begin{document}

\begin{abstract}
Suppose $K \subset S^3$ is a knot and suppose $p$ and $q$ are co-prime integers with $q\ge 1$. For any field $\mathbb{K}$, we establish a dimension formula for the framed instanton homology of knot surgeries:
$$
\dim I^\sharp(S^3_{p/q}(K); \mathbb{K}) = q \cdot r_{\mathbb{K}}(K) + |p - q \cdot \nu^\sharp_{\mathbb{K}}(K)|
$$
for certain integers $r_{\mathbb{K}}(K)$ and $\nu^\sharp_{\mathbb{K}}(K)$, except possibly when $p/q = \nu^\sharp_{\mathbb{K}}(K)$ and $\nu^\sharp_{\mathbb{K}}(K)$ is even. This formula generalizes the result of Baldwin--Sivek from the case $\mathbb{K} = \mathbb{C}$ to arbitrary fields. Based on the result for $\mathbb{K} = \mathbb{Z}/2$, we obtain that $S^3_{p/q}(K)$ is not $SU(2)$-abelian for any knot $K$ other than the unknot and the right-handed trefoil whenever $p/q \in [0,6)$ and $p \in \{ a^e, 2a^e \}$ for some prime number $a$ and natural number $e$, thereby extending existing results for $p/q \in [0,5]$ and $p = a^e$. A byproduct of the techniques developed in this paper is that we generalize the distance-two surgery exact triangle by Culler--Daemi--Xie and Daemi--Miller-Eismeier--Lidman from $\mathbb{Z}/2$ coefficients to any coefficient ring.
\end{abstract}
\maketitle
\section{Introduction}
For a closed connected oriented $3$-manifold $Y$ and an unoriented $1$-submanifold $\omega\subset Y$, Kronheimer--Mrowka \cite{kronheimer2010knots,kronheimer2011knot,kronheimer2011khovanov} constructed a relative $\Z/4$-graded $\Z$-module $I^\sharp(Y,\omega)$, called the \emph{framed instanton homology} of $(Y,\omega)$; see \cite[\S 2.1]{LY2025torsion} for the relation between different definitions. When $\omega=0$, the empty set, the module $I^\sharp(Y,\omega)$ can be regarded as a deformation of the homology of the representation variety (without quotient by conjugation) of $Y$\[R(Y)=\operatorname{Hom}(\pi_1(Y),SU(2)),\]and is closely related to the existence of irreducible $SU(2)$ representations of $\pi_1(Y)$ (i.e.\ $SU(2)$ representations with nonabelian images).

In \cite{baldwin2020concordance,baldwin2022concordanceII}, Baldwin--Sivek established a dimension formula for the framed instanton homology of knot surgeries $I^\sharp(S^3_{p/q}(K),\omega;\C)$ over $\C$, where $K\subset S^3$ is a knot and $p$ and $q$ are co-prime integers, which is similar to that for Heegaard Floer homology over $\F_2=\Z/2$ by Hanselman \cite[Proposition 15]{Hanselman2023cosmetic}. Baldwin--Sivek's work depends heavily on the structure theorem and the adjunction formula for the instanton cobordism maps developed in \cite{baldwin2019lspace}. Later, Deeparaj \cite{bhat2023newtriangle} and the authors of this paper \cite{LY2025torsion,LY20255surgery} studied these instanton homologies over $\F_2$, which turn out to violate the adjunction formula and provides more information about irreducible $SU(2)$ representations.

In this paper, we study the framed instanton homology of knot surgeries over a general field $\bK$, and obtain applications about irreducible $SU(2)$ representations by taking $\bK=\F_2$. It is worth mentioning that many techniques indeed work over any coefficient ring, but we only focus on field coefficients because vector spaces over a field are easier to study. Throughout this paper, we fix a field $\bK$ and write $\dim$ for $\dim_{\bK}$ for short.

Note that $I^\sharp(Y,\omega)$ depends only on the homology class $[\omega]\in H_1(Y;\F_2)$ up to isomorphism, (though the isomorphism is canonical only when $b_1(Y)=0$; see \S \ref{sec: bundle sets}). A direct computation shows that
\begin{equation}\label{eq: homology of surgery}
    H_1(S^3_{p/q}(K);\F_2) = \begin{cases}
    0 & \mathrm{when}~p~\mathrm{odd};\\
    \F_2\langle [\mu]\rangle & \mathrm{when}~p~\mathrm{even},
\end{cases}
\end{equation}
where $\mu$ is the meridian of $K$. Hence the dimension of $I^\sharp(Y,\omega)$ is independent of $\omega$ when $p$ is odd, and has only two possibilities $I^\sharp(Y)=I^\sharp(Y,0)$ and $I^\sharp(Y,\mu)$ when $p$ is even.

The main theorem of this paper is the following dimension formula for the framed instanton homology of knot surgeries.
\begin{theorem}[{Propositions \ref{prop: dim formula for integers} and  \ref{prop: dimension formula}}]\label{thm: dimension formula, main}
Suppose $K \subset S^3$ is a knot and suppose $\mu$ is the meridian of $K$. Suppose $p$ and $q$ are co-prime integers with $q\ge 1$. Then there exists a concordance invariant $\nk{K}\in\Z$ satisfying $\nk{\widebar{K}}=\nk{K}$ for the mirror knot $\widebar{K}$. Moreover, for \begin{equation*}
	    M=\nk{K}\aand R=r_{\bK}(K)=\min \left\{\dim I^\sharp(S^3_{M}(K);\bK),\dim I^\sharp(S^3_{M}(K),\mu;\bK)\right\},
	\end{equation*}we have
	\begin{equation}\label{eq: dim formula rational, main}
	    \dim I^\sharp(S^3_{p/q}(K);\bK)=\dim I^\sharp(S^3_{p/q}(K),\mu;\bK)= qR + |p-qM|.
	\end{equation}
    except possibly when $p/q = M$ and $M$ is even. In the exceptional case, we have
	\begin{equation}\label{eq: differ 2}
	    \left\{\dim I^\sharp(S^3_{M}(K);\bK),\dim I^\sharp(S^3_{M}(K),\mu;\bK)\right\}=\{R,R+2\}.
	\end{equation}
\end{theorem}
\brem\label{rem: inequality}
From \cite[Corollary 1.4]{scaduto2015instantons}, we have\[\chi(I^\sharp(S^3_{p/q}(K);\bK))=\chi(I^\sharp(S^3_{p/q}(K),\mu;\bK))=|p|.\]Hence \[\dim I^\sharp(S^3_{p/q}(K);\bK)=|p|+2k\aand \dim I^\sharp(S^3_{p/q}(K),\mu;\bK)=|p|+2l\]for some $k,l\in\Z_+$ depending on $p/q$. Taking $p/q=\nk{K}$, we obtain that \[r_{\bK}(K)=|\nu^\sharp_{\bK}(K)|+2h\]for some $h\in\Z_+$.
\erem
\bdefn
For a knot $K\subset S^3$, it is called \emph{V-shaped} over $\bK$ if either $\nk{K}$ is odd, or $\nk{K}$ is even and\[r_{\bK}(K)=\dim I^\sharp(S^3_{\nk{K}}(K);\bK).\]It is called \emph{W-shaped} over $\bK$ if $\nk{K}$ is even and \[r_{\bK}(K)=\dim I^\sharp(S^3_{\nk{K}}(K),\mu;\bK).\]
\edefn
\brem\label{rem: other work}
The case $\bK=\C$ in Theorem \ref{thm: dimension formula, main} was proven by Baldwin--Sivek \cite{baldwin2020concordance,baldwin2022concordanceII}. They also showed that $\nu^\sharp_\C(K)$ is even only when it is zero. The case $\bK=\F_2$ in Theorem \ref{thm: dimension formula, main} was also studied independently in upcoming work of Ghosh--Miller-Eismeier \cite{SGMME}. Furthermore, they showed that $K$ is always W-shaped over $\F_2$ and both $\nu^\sharp_{\F_2}(K)$ and $r_{\F_2}(K)$ are divisible by $4$. Those results over $\C$ and $\F_2$ heavily depend on the coefficient fields, while our techniques are more general and potentially applicable to all coefficient rings. In particular, our proof of \eqref{eq: differ 2} (cf.\ Proposition \ref{prop: differed at most by 2}) only involves the $\Z/4$-grading of instanton homologies, which is much simpler than the proofs for $\bK=\C$ in \cite[Theorem 6.1]{baldwin2022concordanceII} and $\bK=\F_2$ in \cite{SGMME}.
\erem

Towards proving the Property P conjecture, Kronheimer--Mrowka \cite{kronheimer04su2} showed that for any nontrivial knot $K$ and any rational slope $r\in[0,2]$, the surgery manifold $S^3_r(K)$ is not $SU(2)$-abelian. Here, a closed $3$-manifold $Y$ is called \emph{$SU(2)$-abelian} if $\pi_{1}(Y)$ admits no irreducible representation into $SU(2)$. As $S^3_{r}(\widebar{K})\cong -S^3_{-r}(K)$ for the mirror knot $\widebar{K}$ of $K$, one can easily extend the set of non-$SU(2)$-abelian slopes to $[-2,2]$. Thus, we only consider positive slopes in the following discussion.

Later, the set of non-$SU(2)$-abelian slopes for any nontrivial knot was extended by many people \cite{baldwin2019lspace,BLSY21,farber2024fixed,sivek2022cyclic}, which also include all rationals $p/q\in (2,5)$ with $p=a^e$ for some prime number $a$ and natural number $e$, all rationals $p/q\in (5,7)$ for $p=2^e$, and all rationals greater than a fixed constant $N_K$ that depends on the knot $K$. From Moser \cite{moser1971elementary}, for the right-handed trefoil $T_{2,3}$, we have\begin{equation*}\label{eq: T23}
    S^3_5(T_{2,3})\cong L(5,4),~S^3_6(T_{2,3})\cong L(2,1)\# L(3,2),\aand S^3_7(T_{2,3})\cong L(7,4),
\end{equation*}which are all $SU(2)$-abelian. Moreover, from Sivek--Zentner \cite[Proposition 4.3]{SZ2022surgery}, for $p/q\in [0,8]$, the manifold $S^3_{p/q}(T_{2,3})$ is $SU(2)$-abelian only for \[\frac{p}{q}=\big\{6,6\pm \frac{1}{n}\big\}_{n\in\Z_+}.\]Recently, the authors of this paper \cite{LY20255surgery} showed that $S_5^3(K)$ is not $SU(2)$-abelian for any nontrivial knot except $T_{2,3}$.

The main applications of this paper are the following theorems. Recall from \cite{baldwin2019lspace} that a knot $K\subset S^3$ is called an \emph{instanton L-space knot} over a field $\bK$ if there exists $p/q\in \Q_+$ such that \[\dim I^\sharp(S^3_{p/q}(K);\bK)=|p|.\]
\bthm[{Theorem \ref{thm: larger than nu}}]\label{thm: larger than nu, main}
Suppose $K$ is a nontrivial knot of genus $g~(\ge 1)$ and suppose $p/q\in (0,\infty)$ with $q\ge 1$, $\gcd(p,q)=1$, and $p\in\{a^e,2a^e\}$ for some prime number $a$ and natural number $e$. If $S^3_{p/q}(K)$ is $SU(2)$-abelian, then the knot $K$ is an instanton L-space knot over any field $\bK$,\[r_{\bK}(K)=\nk{K}\ge \nu^\sharp_\C(K)=2g-1\aand p/q\ge \nk{K}.\]Moreover, for $\bK=\F_2$, we have\[\nu^\sharp_{\F_2}(K)\ge \nu^\sharp_\C(K)+1=2g.\]
\ethm
% \bcor\label{cor: 2g, main}
% Suppose $K$ is a nontrivial knot with genus $g$. If $g$ is a prime power, then $S^3_{2g}(K)$ is not $SU(2)$-abelian.
% \ecor
\bthm[{Theorem \ref{thm: abelian}}]\label{thm: abelian main}
Suppose $K$ is a nontrivial knot and suppose $p/q\in (2,6)$ with $q\ge 1$, $\gcd(p,q)=1$, and $p\in\{a^e,2a^e\}$ for some prime number $a$ and non-negative integer $e$. Then $S^3_{p/q}(K)$ is $SU(2)$-abelian only when $K=T_{2,3}$ and \[\frac{p}{q}\in \big\{6-\frac{1}{n}\big\}_{n\in\Z_+}.\]
\ethm
\brem
%Together with Ghosh--Miller-Eismeier's work \cite{SGMME} mentioned in Remark \ref{rem: other work}, one could extend the set of non-$SU(2)$-abelian slopes to larger rationals. The W-shaped fact would imply $p/q>\nu^\sharp_{\F_2}(K)$ in the assumption of Theorem \ref{thm: larger than nu, main} and hence extend Theorem \ref{thm: abelian main} to $p/q=6$ with the right-handed trefoil as the only exception. Note that the assumption $p\in\{a^e,2a^e\}$ is about the nondegeneracy of the generators of instanton homology (cf.\ the proof of Lemma \ref{lem: nondegenerate}) and the case of $2a^e$ is overlooked by Baldwin--Sivek \cite{baldwin2019lspace} but important for $p/q=6$. Similarly, one could also use the W-shaped fact to show that $S^3_{2g}(K)$ is always not $SU(2)$-abelian when $g=g(K)$ is a prime power.

%Moreover, the divisibility by $4$ for $\bK=\F_2$ would further extend the set of non-$SU(2)$-abelian slopes to $(6,8)$, with the assumption $p\in\{a^e,2a^e\}$ and the exception of the right-handed trefoil and $p/q=7$. Together with the W-shaped fact, one could also include $p/q=8$. Note that $S^3_8(T_{2,3})$ is not $SU(2)$-abelian (cf.\ \cite[Proposition 4.3]{SZ2022surgery}). Ghosh--Miller-Eismeier \cite{SGMME} stated the result for $S_7^3(K)$, while we notice the nondegeneracy condition for $p=2a^e$ and include $S_6^3(K)$. Note that $S_8^3(K)$ is already included in the case $p=a^e$.

We would like to remark that Baldwin--Sivek \cite{baldwin2019lspace} only discussed the non-degeneracy for the case $p=a^e$ for some prime number $a$ and non-negative integer $e$, but their argument applies verbatim to $p=2a^e$ as well (cf.\ Lemma \ref{lem: nondegenerate}). This latter case matters now because the integral slope $6=2\times 3$ is precisely of this form.

%We would like to also remark that the interval $(2,6)$ in Theorem \ref{thm: abelian main} could be further extended to $(2,8]$ with the results established in Ghosh--Miller-Eismeier's work \cite{SGMME}. In their paper, they could show that over $\bK=\mathbb{F}_2$, all knots are W-shaped and $\nu^{\sharp}_{\mathbb{F}_2}$ and $r_{\mathbb{F}}^{\sharp}$ are both divisible by $4$. Thus, if $K$ is a knot and its $p/q$-surgery is $SU(2)$-abelian, with $p/q\in [6,8]$ and $p\in \{a^e, 2a^e\}$ (note $6$ and $8$ are both included), then Theorem \ref{thm: larger than nu, main} suggests that $p/q\geq \nu^{\sharp}_{\mathbb{F}_2} = r_{\mathbb{F}_2}^{\sharp}$. Thus, divisibility by $4$ would force either $\nu^{\sharp}_{\mathbb{F}_2} = 4$, which implies that $K$ is the right-handed trefoil; or $p/q=\nu^{\sharp}_{\mathbb{F}_2} = 8$. Then, the fact that $K$ must be W-shaped would further excludes the possibility for $8$-surgery to be $SU(2)$-abelian.

Through private communication, we were told that Ghosh and Miller-Eismeier, in their upcoming work \cite{SGMME}, could show that over $\bK=\mathbb{F}_2$, all knots are W-shaped and $\nu^{\sharp}_{\mathbb{F}_2}$ and $r_{\mathbb{F}}^{\sharp}$ are both divisible by $4$, and as an application, they conclude that for any non-trivial knot which is not $T_{2,3}$, the $7$-surgery must be non-$SU(2)$-abelian. We would like to remark that with the help of their results, the interval $(2,6)$ in Theorem \ref{thm: abelian main} could be further extended to $(2,8]$, with the exception $T_{2,3}$ and $p/q\in\{6,6\pm 1/n\}_{n\in\Z_+}$. Indeed, if $K$ were a knot and its $p/q$-surgery is $SU(2)$-abelian, with $p/q\in [6,8]$ and $p\in \{a^e, 2a^e\}$ (note $6$ and $8$ are both included), then Theorem \ref{thm: larger than nu, main} would suggest that $p/q\geq \nu^{\sharp}_{\mathbb{F}_2} = r_{\mathbb{F}_2}^{\sharp}$. Thus, divisibility by $4$ would force either $\nu^{\sharp}_{\mathbb{F}_2} = 4$, which would imply that $K=T_{2,3}$; or $p/q=\nu^{\sharp}_{\mathbb{F}_2} = 8$. Then, the fact that $K$ must be W-shaped would further exclude the possibility for $8$-surgery to be $SU(2)$-abelian.

%We would like to also remark that the interval $(2,6)$ in Theorem \ref{thm: abelian main} could be further extended to $(2,8]$ if one could show that over $\bK=\mathbb{F}_2$, all knots are W-shaped and $\nu^{\sharp}_{\mathbb{F}_2}$ is divisible by $4$. If these were true, and if $K$ were a knot and its $p/q$-surgery is $SU(2)$-abelian, with $p/q\in [6,8]$ and $p\in \{a^e, 2a^e\}$ (note $6$ and $8$ are both included), then Theorem \ref{thm: larger than nu, main} would imply that $p/q\geq \nu^{\sharp}_{\mathbb{F}_2} = r_{\mathbb{F}_2}^{\sharp}$. Thus, divisibility by $4$ would further force either $\nu^{\sharp}_{\mathbb{F}_2} = 4$, which would then imply that $K$ is the right-handed trefoil; or $p/q=\nu^{\sharp}_{\mathbb{F}_2} = 8$. Then, the fact that $K$ must be W-shaped would further exclude the possibility for $8$-surgery to be $SU(2)$-abelian.

%During our pursuit of always-W-shape and divisibility-by-4 results, we were told that Ghosh Miller-Eismeier had already established these two facts and will include them in their up coming paper in \cite{SGMME}. But since their paper hasn't been ready yet, we decide to only include these discussions as a remark.
\erem

Another application of Theorem \ref{thm: dimension formula, main} is to provide a better bound for the limiting slope of $SU(2)$-averse knot, introduced by Sivek--Zentner in \cite{sivek2022cyclic}. They showed that if a knot $K\subset S^3$ admits infinitely many $SU(2)$-cyclic surgeries (which they call a \emph{$SU(2)$-averse knot}), then the set of such slopes has a unique limiting point which is a rational number and is denoted by $r(K)$ (cf. \cite[Theorem 1.1]{sivek2022cyclic}). Moreover, the proof of \cite[Theorem 9.1]{sivek2022cyclic} states that when $r(K)\ge 0$, there are infinitely many slopes with prime numerators in $[\lceil r(K) \rceil - 1, \lceil r(K) \rceil]$ which produces $SU(2)$-abelian 3-manifolds, where $\lceil x \rceil$ denotes the minimal integer no less than $x$. As $r(\widebar{K})=-r(K)$ for the mirror knot $K$, Theorem \ref{thm: abelian main} implies the following corollary.
\bcor
Suppose $K\subset S^3$ is an $SU(2)$-averse knot and suppose $\mathbb{K}$ is an arbitrary field. Then we have
\[
\lceil |r(K)| \rceil \geq |\nu^{\sharp}_{\mathbb{K}}(K)| + 1.
\]In particular, by taking $\bK=\F_2$, we have\[\lceil |r(K)| \rceil \ge 2g+1.\]
\ecor
% \bpf
% By passing to the mirror if necessary, we can assume that $r(K)\geq 0$. Then the proof of \cite[Theorem 9.1]{sivek2022cyclic} states that there are infinitely many slopes in $[\lceil r \rceil - 1, \lceil r \rceil]$ which produces SU(2)-cyclic 3-manifolds and whose numerators are primes. Thus for any such slope $\frac{p}{q}$ we conclude
% \[
%     \dim I^{\#}(S^3_r(K);\mathbb{K}) = p
% \]
% and Theorem 1.1 implies that $\lceil r \rceil - 1 \geq \nu^{\sharp}_{\mathbb{K}}$.
% \epf

Finally, we consider the genus-one knots and propose some questions.

\bprop[Proposition \ref{prop: genus one proof}]\label{prop: genus 1}
Suppose $K\subset S^3$ is a genus-one knot and suppose $\bK$ is a field with $\mathrm{char}(\bK)\neq 2$. Then \[|\nu^\sharp_{\bK}(K)|\le 1.\]
\eprop

\begin{example}
    For $n\in \Z_+$, let $K_{n}$ be the twist knot with $n$ positive half-twist. Note that all $K_n$ are genus-one. Note that $K_1$ is the left-handed trefoil and $K_2$ is the figure-eight knot. From \cite[Theorem 1.13]{baldwin2020concordance}, \cite[Corollaries 1.6 and 1.7]{scaduto2015instantons}, and Proposition \ref{prop: genus 1}, for any field $\bK$ with $\mathrm{char}(\bK)\neq 2$, we have\[r_{\C}(K_{2n-1})=2n-1,~\nu^\sharp_{\C}(K_{2n-1})=-1,~r_{\C}(K_{2n})=2n,~\nu^\sharp_{\C}(K_{2n})=0,\]\[\left(r_{\bK}(K_{2n-1}),\nu^\sharp_{\bK}(K_{2n-1})\right)\in\{(2n-1,-1),(2n,0),(2n+1,1)\},\]\[\aand \left(r_{\bK}(K_{2n}),\nu^\sharp_{\bK}(K_{2n})\right)\in\{(2n-1,1),(2n,0),(2n+1,-1)\}.\]    
    % From \cite[Theorem 1.10]{SGMME}, we know that\[r_{\F_2}(K_{2n-1})=8n-4,~\nu^\sharp_{\F_2}(K_{2n-1})=-4,~r_{\F_2}(K_{2n})=8n,\aand \nu^\sharp_{\F_2}(K_{2n})=0.\]
\end{example}
Motivated by the above example, it is natural to ask the following question.
\begin{ques}\label{ques: same for all coefficients}
    For any field $\bK$ with $\mathrm{char}(\bK)\neq 2$ and any knot $K\subset S^3$, do we always have the following equations?\[r_{\bK}(K)=r_{\C}(K)\aand \nu^\sharp_\bK(K)=\nu^\sharp_\C(K).\]
\end{ques}
\brem
By footnote of \cite[p. 26]{baldwin2019lspace}, the structure theorem of the cobordism map and the generalized eigenspace decomposition of the instanton homology are expected to hold over any field $\bK$ with $\mathrm{char}(\bK)=0$. By applying techniques in \cite{baldwin2020concordance,baldwin2022concordanceII}, Question \ref{ques: same for all coefficients} might have a positive answer for all fields with $\mathrm{char}(\bK)=0$.
\erem
Motivated by the case $\bK=\F_2$ in \cite{SGMME} and the fact that the figure-eight knot is W-shaped over $\C$ \cite[Theorem 10.4]{baldwin2022concordanceII}, we propose the following question.
\begin{ques}\label{ques: W-shaped}
    If $\nu^\sharp_{\bK}(K)$ is even for some knot $K\subset S^3$ and some field $\bK$, then do we always have that $K$ is W-shaped over $\bK$?
\end{ques}

\subsection{Distance-two surgery exact triangle}

The proof of Theorem \ref{thm: dimension formula, main} relies on some commutative diagrams of instanton cobordism maps from embedded spheres of self-intersection $-1$ and $-2$, where the one for $-1$ follows from the blow-up formula for the Donaldson invariant \cite{Donaldson1990polynomial,DK1992instanton,Ozsvath1994blowup}, and the one for $-2$ is an ingredient in the proof of the \emph{distance-two surgery triangle} over the coefficient field $\F_2$ by Culler--Daemi--Xie \cite[Theorem 1.6 for $N=2$]{CDX2020polygon} and Daemi--Miller-Eismeier--Lidman \cite[Theorem 1.12]{DMML2024distancetwo}. Based on those commutative diagrams, together with the original Floer's surgery exact triangle \cite{floer1990knot,scaduto2015instantons}, we generalize the distance-two surgery triangle to any coefficient ring. To describe the exact triangle, we first introduce the following setup.
\bdefn\label{defn: surgery tuple}
Let $Y$ be a closed connected oriented $3$-manifold and let $\omega\subset Y$ be an unoriented $1$-submanifold. Suppose $K\subset Y$ is a framed knot. We write $Y\bbslash K=Y\backslash{\rm int}N(K)$. Suppose $\mu$ is the meridian of $K$ and $\lambda$ is the framed longitude of $K$, which are both on $\partial (Y\bbslash K)$ and satisfy $\mu\cdot \lambda=-1$. We call $(Y,\omega,K)$ a \emph{surgery tuple} if either of the following conditions hold. 
\begin{itemize}
    \item $(Y,\omega)$ is \emph{nontrivial admissible}, i.e.\ there exists a closed embedded oriented surface $\Sigma\subset Y$ such that the algebraic intersection number $\omega\cdot\Sigma$ is odd. The knot $K$ is disjoint from $\omega$ and $\Sigma$. In this case, we also call $(Y,\omega,K)$ \emph{nontrivial admissible surgery tuple}.
    \item $(Y,\omega)$ is \emph{trivial admissible}, i.e.\ $Y$ is a homology sphere. The knot $K$ is framed by the boundary of a Seifert surface and $\omega\in\{0,\lambda\}$. In this case, we also call $(Y,\omega,K)$ \emph{trivial admissible surgery tuple}.
\end{itemize}
We call the pair $(Y,\omega)$ \emph{admissible} if it is either nontrivially or trivially admissible.

For a surgery tuple, one can consider the surgery manifold $Y_{p/q}(K)$ obtained from $Y$ by $p/q$-surgery on $K$ with respect to the basis $(\mu,\lambda)$. Note that $\mu$ and $\lambda$ also lie in $Y_{p/q}(K)$, and we write $\mu+\lambda$ for the curve obtained from $\mu\cup \lambda$ by resolving the unique intersection point. 

Let $\wti{K}_{p/q}\subset Y_{p/q}(K)$ be the dual knot in the surgery manifold, i.e.\ the core of the Dehn filling solid torus. Note that for nontrivial admissible surgery tuple $(Y,\omega,K)$, the pair $(Y_{p/q}(K),\omega,\wti{K}_{p/q})$ is again a nontrivial admissible surgery tuple. For a trivial admissible surgery tuple $(Y,\omega,K)$, the pair $(Y_{1/q}(K),\omega,\wti{K}_{1/q})$ is again a trivial admissible surgery tuple, and $(Y_0(K),\mu)$ is a nontrivial admissible pair with the surface $\Sigma$ being the cap-off of the Seifert surface of $K$.

For surgery slopes $p_1/q_1$ and $p_2/q_2$ such that $|p_1q_2-p_2q_1|=1$, we write $W^{p_1/q_1}_{p_2/q_2}$ for the elementary cobordism from the surgery, called the \emph{surgery cobordism}. Let $\Dcc{p_1/q_1}{p_2/q_2}$ and $\Dc{p_1/q_1}{p_2/q_2}$ be the cocore disk and the core disk in $W^{p_1/q_1}_{p_2/q_2}$.
\edefn
\brem
Some authors use the notation \emph{admissible bundles} only for nontrivial admissible pairs. Here we follow \cite{scaduto2015instantons,CDX2020polygon} and also include the case of homology spheres in the admissible pair, and call them trivial admissible pairs.
\erem
\bthm[Theorem \ref{thm: distance two}]\label{thm: distance two triangle CDX}
Suppose $(Y,\omega,K)$ is a surgery tuple as in Definition \ref{defn: surgery tuple}. Let $\cR$ be any coefficient ring. Then there exists an exact triangle\begin{equation*}
	\xymatrix{
	I(Y_{-1}(K),\omega;\cR)\ar[rr]^{h}&& I(Y_{1}(K),\omega\cup \mu;\cR)\ar[dl]^{(f_1,f_2)}\\
	&I(Y,\omega;\cR)\oplus I(Y,\omega\cup \lambda;\cR)\ar[lu]^{g_1+g_2}&
	}
\end{equation*}Moreover, the maps $f_1,f_2,g_1,g_2,h$ are instanton cobordism maps with suitable signs (we omit $\cR$)\[f_1=I(W^{1}_{\infty},(\omega\times I)\cup \Dcc{1}{\infty}),~f_2=I(W^{1}_{\infty},(\omega\times I)\cup \Dcc{1}{\infty}\cup \Dcc{1}{\infty}),\]\[g_1=I(W^\infty_{-1},\omega\times I),~g_2=I(W^\infty_{-1},(\omega\times I)\cup \Dc{\infty}{1}),\]\[\aand h:I(W^{0}_{1},(\omega\times I)\cup\Dc{0}{1}\cup \Dcc{0}{1})\circ (W^{-1}_{0},(\omega\times I)\cup\Dcc{-1}{0}).\]
\ethm

\brem\label{rem: octahedral trick}
Our proof of Theorem \ref{thm: distance two triangle CDX} does not rely on the usual triangle detection lemma \cite[Lemma 7.1]{kronheimer2011khovanov}, but uses diagram chasing as in \cite[\S 5]{LY2022integral1} (especially the proof of \cite[Proposition 5.3]{LY2022integral1}). Hence we do not need maps for cobordisms with families of metrics, although they are important for the applications of the distance-two triangle in \cite{DMML2024distancetwo}. If one cares only about the existence of the maps $f_1,f_2,g_1,g_2$, then there is a simpler proof based on the octahedral lemma \cite[Lemma A.3.10]{OSS2015grid}, which does not depend on the commutative diagrams about embedded spheres of self-intersection $-1$ and $-2$. This situation is similar to that in \cite[Remark 3.7]{LY2022integral1}. Finally, note that the map $h$ in \cite{scaduto2015instantons,CDX2020polygon} is from the cobordism map with ``middle end" $\mathbb{RP}^3$, which is expected to be the same as the one we used in a neck-stretching argument along $\mathbb{RP}^3$ and the analysis in the last paragraph of the proof of \cite[Proposition 5.11]{DMML2024distancetwo}.
\erem
\begin{org}
    In \S \ref{sec: bundle sets}, we review the dependence of the bundle sets for instanton Floer homology. In \S \ref{sec: triangle}, we revisit the surgery exact triangle for instanton homology, focusing on the bundle sets of the surgery cobordism maps. In \S \ref{sec: embedded sphere}, we consider the embedded spheres of self-intersection $0,-1,-2$ in the cobordisms and study their effects on instanton cobordism maps. We also prove Theorem \ref{thm: distance two triangle CDX} as a byproduct. The results in \S \ref{sec: bundle sets}-\ref{sec: embedded sphere} work over any coefficient ring $\cR$, while in remaining sections we focus on a field $\bK$. In \S \ref{sec: integer surgeries} and \S \ref{sec: rational surgeries}, we deal with the integral and rational cases of Theorem \ref{thm: dimension formula, main} separately. In \S \ref{sec: SU(2)-representations}, we study the connection between $SU(2)$-representations and instanton homology, and prove Theorems \ref{thm: larger than nu, main} and \ref{thm: abelian main}. In \S \ref{sec: genus 1}, we fix the proof of the instanton bypass exact triangle in \cite[\S 4]{BS2022khovanov}. As a byproduct, we consider the genus-one knots and prove Proposition \ref{prop: genus 1}.
\end{org}
\begin{ack}
The authors thank Sudipta Ghosh and Mike Miller Eismeier for sharing the draft of their paper \cite{SGMME}. The authors also thank John Baldwin, Steven Sivek, and Christopher Scaduto for email communications about bundle sets. The second author is partially supported by Simons Collaboration \#271133 from Peter Kronheimer and is also grateful to Yi Liu for the invitation to BICMR, Peking University while this project is developed. 
\end{ack}

\section{Bundle sets of instanton Floer homology}\label{sec: bundle sets}
In this section, we discuss the bundle sets for instanton Floer homology. By default, every manifold is smooth. We fix a commutative coefficient ring $\cR$ and write $\Mod_\cR$ for the category of $\cR$-modules. We follow the setup in \cite[\S 4.1]{kronheimer2011khovanov} and consider the category $\operatorname{W}$ described as follows.
\begin{itemize}
    \item The objects are nontrivial admissible pairs $(Y,\omega)$ from Definition \ref{defn: surgery tuple}, where $\omega$ is called the \emph{bundle set} in $Y$.
    \item The morphisms are isomorphism classes of pairs $(W,\nu)$, where $W$ is an oriented connected cobordism between the $3$-manifolds in nontrivial admissible pairs and $\nu$ is an unoriented embedded $2$-dimensional simplicial complex inside $W$, serving as a cobordism between bundle sets in nontrivial admissible pairs. We will call $(W,\nu)$ a cobordism between nontrivial admissible pairs, or sometimes simply a \emph{cobordism}. We also call $\nu$ the \emph{bundle set} in $W$.
\end{itemize} 

Instanton Floer theory induces a projective functor\[I(-;\cR):\operatorname{W}\to \Mod_\cR,\]where \emph{projective} means the images of objects and morphisms of the functor are only well-defined up to sign, or, equivalently, we consider the unordered pair of objects and morphisms $\{h,-h\}$ in $\Mod_\cR$. The module $I(Y,\omega;\cR)$ for a nontrivial admissible pair $(Y,\omega)$ is called the \emph{instanton (Floer) homology} of $(Y,\omega)$ and the projective morphism $I(W,\nu;\cR)$ is called an \emph{(instanton) cobordism map}. We will omit $\cR$ in the cobordism map but keep it for instanton homology to indicate the dependence. Here we can replace $\Mod_\cR$ by categories of absolute $\Z/2$-graded or relative $\Z/8$-graded $\cR$-modules. There is also instanton homology for a trivial admissible pair from Definition \ref{defn: surgery tuple}, and cobordism maps between (either nontrivial or trivial) admissible pairs.

The \emph{framed instanton homology} $I^\sharp(Y,\omega;\cR)$ is defined by taking the connected sum of $(Y,\omega)$ with \begin{itemize}
    \item either $((S^3,H),\alpha)$ for a Hopf link $H$ as a singular set and an trivial arc $\alpha$ connecting two components of $H$,
    \item or $(T^3,S^1)$ with larger gauge group corresponding to the dual torus $R$ of $S^1$,
\end{itemize}where $Y$ can be any closed oriented connected $3$-manifold and $\omega$ can be any unoriented $1$-submanifold; see \cite[\S 2.1]{LY2025torsion} for the relation between different definitions. It also has similar properties as $I(Y,\omega;\cR)$, but only relatively $\Z/4$-graded; see \S \ref{subsec: Z/4} for more discussion and applications.

It is worth mentioning that the construction of instanton homology for $(Y,\omega)$ needs an actual $SO(3)$-bundle $P\to Y$, but the pair $(Y,\omega)$ only determines a transitive system of bundles: different bundles are related by canonical cobordisms over $Y\times I$, which lead to canonical isomorphisms (up to signs) for the associated instanton homology. Thus, the instanton homology $(Y,\omega)$ is only well-defined up to sign.

The sign ambiguity can be fixed by further picking a homology orientation or an almost complex structure on the cobordism between admissible pairs, but we are satisfied with the projective property. We can pick representatives in the unordered pair $\{h,-h\}$ of objects and morphisms, at the cost that all equations between cobordism maps only hold up to a sign. 

Recall that $\omega$ in an admissible pair $(Y,\omega)$ is indeed a geometric representative of the Poincar\'{e} duality of the second Stiefel--Whitney class $\omega_2(P)\in H^2(Y;\F_2)$ of the $SO(3)$ bundle $P\to Y$. Hence the isomorphism class of $I(Y,\omega;\cR)$ depends only on the homology class $[\omega]\in H_1(Y;\F_2)$. However, when considering cobordism maps, we need to pin down actual modules instead of just isomorphism classes. We make explicit the effect of changing geometric representatives within the same mod $2$ homology class as follows. We first deal with the dependence of the bundle sets in cobordisms.

\begin{lemma}\label{lem: cob map depending on H_2}
	Let\[(W,\nu_i): (Y_1,\omega_1)\to (Y_2,\omega_2)\text{ for }i=1,2\]be two cobordisms between the same admissible pairs. If \begin{equation}\label{eq: homology condition}
	    [\nu_1 \cup \nu_2] = 0\in H_2(W;\F_2),
	\end{equation}
	\[
		I(W,\nu_1) = \pm I(W,\nu_2): I(Y_1,\omega_1; \cR) \to I(Y_2,\omega_2;\cR).
	\]
\end{lemma}
\bpf
From \cite[\S 4.1]{kronheimer2011khovanov}, the mod $2$ homology class as the dual of the (relative) second Stiefel--Whitney class determines the $SO(3)$ bundle on $W$ up to bundle isomorphisms and addition of instantons and yields cobordism maps up to a sign. Note that the (relative) second Stiefel--Whitney class itself does not determine the bundle up to bundle isomorphisms because the first Pontryagin class could be different.
\epf
\brem\label{rem: dependence wrong}
In \cite[\S 4]{BS2022khovanov} and some subsequent work (e.g.\ \cite[Remark 2.11]{baldwin2019lspace} and \cite[Remark 2.5]{alfieri2020framed}), the authors claimed the cobordism map $I(W,\nu)$ depends only on the homology class $[\nu]\in H_2(W,\partial W;\F_2)$ up to sign, which is indeed not true in general; see Example \ref{exmp: dependence}. In particular, the proof of the bypass exact triangle in \cite[\S 4]{BS2022khovanov} needs to be fixed; see \S \ref{sec: genus 1}. From the long exact sequence associated to the pair $(W,\partial W)$\begin{equation*}\label{eq: long exact for pair}
    H_2(\partial W;\F_2)\to H_2(W;\F_2)\to H_2(W,\partial W;\F_2)
\end{equation*}
The condition \eqref{eq: homology condition} implies that $[\nu_1]=[\nu_2]\in H_2(W,\partial W;\F_2)$, but the converse is not true in general if $H_2(\partial W;\F_2)\neq 0$. Note that $H_2(\partial W;\F_2)=0$ if and only if $b_1=0$ for both components of $\partial W$.

To fix the sign, one needs to choose some homology orientation of $W$ and the bundle set $\nu$ should be oriented and represents a class $H_2(W,\partial W;\Z)$. Similarly, the condition $[\nu_1\cup \nu_2]=0\in H_2(W;\Z)$ instead of $[\nu_1]=[\nu_2]\in H_2(W,\partial W;\Z)$ implies the identification of the cobordism map. For $\nu_1$ and $\nu_2$ on the same $W$ with the same homology orientation and $[\nu_1\cup (-\nu_2)]=2e\in H_2(W;\Z)$, we have\[I(W,\nu_1)=(-1)^{e\cdot e}I(W,\nu_2).\]See the end of \cite[\S 2.3]{KM2022relation}, \cite[Corollary (3.28)]{Donaldson1987orientation}, and \cite[Remark 6.2]{baldwin2019lspace}. In this paper, for simplicity, we will only consider the mod $2$ class and the cobordism maps up to signs.
\erem

Next, we handle the situation where the bundle set on the $3$-manifolds varies within its mod $2$ homology class.
\begin{lemma}\label{lem: null-homologous curves bound embedded surfaces}
	Suppose $Y$ is a closed connected oriented $3$-manifold and $\omega\subset Y$ is unoriented $1$-submanifold such that $[\omega] = 0 \in H_1(Y;\F_2)$. Then there exists an embedded, possibly non-orientable surface $S\subset Y$ such that $\partial S = \omega$.
\end{lemma}
\begin{remark}
	This lemma appears to be well known but we were not able to find an explicit proof in the literature to cite. For example, \cite[Theorem 8.3]{conner1979differentiable} indicates the existence of a singular surface but not for embedded surfaces; \cite[Lemma 3.6]{hatcher2007notes} handles the $\mathbb{Z}$ coefficient but not $\mathbb{F}_2$; \cite{boden2022GLPairing} states the precise result but without a proof. Hence we decide to include a proof of the lemma for the sake of completeness of the paper. 
\end{remark}
\begin{proof}[Proof of Lemma \ref{lem: null-homologous curves bound embedded surfaces}]
	We adapt the proof of \cite[Lemma 3.6]{hatcher2007notes}. We triangulate $Y$ so that $\omega$ is in the 1-skeleton. The fact that $[\omega] = 0 \in H_1(Y;\F_2)$ then implies that there exists a $2$-chain $\sigma$ such that
	\[
		\partial \sigma = \omega + 2\cdot \theta
	\]
	for some $1$-chain $\theta$. We then perturb $\sigma$ to obtain a surface $S$ as follows.
	\begin{itemize}
		\item For a $2$-simplex $f$ in $\sigma$, assume that its coefficient is $\lambda \neq 0$. We then take $|\lambda|$ copies of $f$, pushing the interiors of these copies of $f$ disjoint from each other, while keeping the boundary (1-chain) unchanged and still coinciding with $\partial f$.
		\item For a $1$-simplex $e$ in $\sigma$, we pair the two adjacent copies of the (perturbed) 2-simplices and push the interior of the common boundary off the $1$-simplex and make them pairwise disjoint. Note if $e$ is in $\omega$ then it is involved in an odd number of $2$-simplices in $\sigma$, and after pairing, only one is left. If $e$ is not in $\omega$, all $2$-simplices having $e$ as part of their boundary can be paired. Let the new complex after perturbations be $\sigma'$.
		\item For a vertex $v$ involved in $\sigma$, we can take a small ball in $Y$ centered at $v$. The boundary $S^2$ intersects $\sigma'$ at either a (possibly disconnected) simple closed curve if $v$ is not in $\omega$, or a disjoint union of a simple closed curve and a connected single arc if $v$ is in $\omega$. A standard argument enables us to iteratively cut $\sigma'$ along an innermost circle and glue back the disk it bounds (and then push the disk off the sphere $S^2$). The arc component will be left untouched. The final result is an embedded surface $S$ with $\partial S = \omega$, and we are done.
	\end{itemize}
\end{proof}

Now, suppose we have two admissible pairs $(Y,\omega_1)$ and $(Y,\omega_2)$ such that
\[
	[\omega_1] = [\omega_2] \in H_1(Y;\F_2).
\]
Then, by Lemma \ref{lem: null-homologous curves bound embedded surfaces}, there exists an embedded, possibly non-orientable surface $S\subset Y$ such that $\partial S = \omega_1\cup\omega_2$. Inside the product cobordism $Y\times I$, we can make $S$ into a properly embedded surface $\nu_S$ and hence obtain a cobordism \[(Y\times I,\nu_S): (Y,\omega_1) \to (Y,\omega_2).\]Thus we obtain a map
\[
\bI_S=I(Y\times I,\nu_S): I(Y,\omega_1;\cR) \to I(Y,\omega_2;\cR).
\]
\begin{lemma}\label{lem: I_S is an iso}
	Under the above setup, we have the following.
	\begin{itemize}
		\item For any choice of $S$, the map $\bI_S$ is an isomorphism.
		\item If $H_2(Y;\F_2) = 0$ (or equivalently $b_1(Y)=0$), then, up to sign, $\bI_S$ is independent of the choice of $S$. As a consequence, up to a sign, $I(Y,\omega_1;\cR)$ and $I(Y,\omega_2;\cR)$ are canonically isomorphic.
	\end{itemize}
\end{lemma}
\bpf
The surface $S$ also induces a cobordism \[(Y\times I,\nu'_S):(Y,\omega_2)\to (Y,\omega_1)\]and we write $\bI_S'$ for the corresponding cobordism map. It is straightforward to check that \[\omega_1\times [0,2]\cup\nu_S\cup\nu'_S\]represents the trivial homology class in $H_2(Y\times [0,2];\F_2)\cong H_2(Y;\F_2)$. By Lemma \ref{lem: cob map depending on H_2}, we know the composition \[\bI_S'\circ \bI_S=\pm \id.\]Similarly, we know that \[\bI_S\circ \bI_S'=\pm \id.\]Hence $\bI_S$ must also be an isomorphism.

If further we have $H_2(Y;\F_2) = 0$, then $H_2(Y\times I;\F_2) = 0$ and Lemma \ref{lem: cob map depending on H_2} applies again to show that, up to sign, $\bI_S$ is independent of the choice of $S$.
\epf
\brem\label{rem: b1>0}
In the case of $b_1(Y)>0$, suppose $R$ is an embedded, possibly non-orientable surface in $Y$ representing a nontrivial homology class in $H_2(Y;\F_2)$. Let $\omega\subset R$ be a separating $1$-submanifold such that $R\backslash\omega=R_1\cup R_2$. Then $[\omega]=0\in H_1(Y;\F_2)$, but the isomorphisms \[\bI_{R_1},\bI_{R_2}:I(Y,\omega;\cR)\to I(Y,0;\cR)\]are not necessarily the same up to sign. The best result we obtained from Lemma \ref{lem: cob map depending on H_2} is\[(\bI_{R_1}\circ \bI_{R_2}')^2=\pm \id\aand (\bI_{R_1}\circ \bI_{R_1}')^2=\pm \id,\]where \[\bI_{R_1}',\bI_{R_2}':I(Y,0;\cR)\to I(Y,\omega;\cR).\]
\erem

We conclude this section with an example that is related to Remark \ref{rem: dependence wrong}.
\begin{example}\label{exmp: dependence}
	We present an explicit example for which we have two bundle sets $\nu_1,\nu_2$ in some cobordism $W$ such that
	\[
		[\nu_1] = [\nu_2] \in H_2(W,\partial W;\F_2).
	\]
	yet they lead to non-identical cobordism maps even allowing sign ambiguity. We take $\cR=\C$.

	Suppose $K\subset S^3$ is the torus knot $T_{2,5}$. It is a genus-two L-space knot in the sense of \cite{baldwin2019lspace}. Indeed, one may construct more examples from instanton L-space knots of genera larger than two. We consider the surgery cobordism associated to the $0$-surgery on $K$
	\[
		W=W^{\infty}_0: S^3\to S^3_{0}(K)
	\]and take $\wti{\nu}_0$ from \cite[Formula 3.3]{baldwin2022concordanceII} as the bundle set. We consider the cobordism map\[I^\sharp(W,\wti{\nu}_0):I^\sharp(S^3)\to I^\sharp(S^3_0(K),\mu),\]where $\mu$ is a meridian of $K$. Indeed, from Remark \ref{rem: BS notation}, we know that $\wti{\nu}_0$ is the cocore disk, though here we do not use this fact.
    
    Let $\widebar{\Sigma}\subset W$ be the closed oriented surface that is the union of a genus-two Seifert surface of $K$ and a core disk. Note that it is also isotopic to the cap-off of the Seifert surface $\widebar{\Sigma}'$ in $S_0^3(K)$. Hence $\widebar{\Sigma}\cdot \widebar{\Sigma} = 0$, where $\cdot$ denotes the pairing of $H_2(W,\partial W;\Z)$ and $H_2(W;\Z)$.

    Note that $\widebar{\Sigma}$ represents a generator of\[\Z\cong H_2(S_0^3(K);\Z)\xra{\cong}H_2(W;\Z).\]Let\[
		\begin{aligned}
			s_i: H_2(W;\mathbb{Z}) &\to 2\mathbb{Z}\\
			[\widebar{\Sigma}]&\to 2i.
		\end{aligned}
	\]From \cite[Theorem 1.16]{baldwin2019lspace} and the facts that $b_1(W)=0$, $g(\widebar{\Sigma})=2$ and $\widebar{\Sigma}\cdot \widebar{\Sigma} = 0$, we obtain a decomposition of the cobordism map\begin{equation}\label{eq: decomposition}
	    I^\sharp(W,\wti{\nu}_0)=\sum_{i=-1}^1 I^\sharp(W,\wti{\nu}_0;s_i),
	\end{equation}where the image of $I^\sharp(W,\wti{\nu}_0;s_i)$ lies in the $(2,2i)$-generalized eigenspace of the actions $(\mu(\pt),\mu(\widebar{\Sigma}'))$. Moreover, from \cite[Theorem 1.18]{baldwin2020concordance} and \cite[Proposition 4.3]{baldwin2022concordanceII}, we know that each term in the decomposition \eqref{eq: decomposition} is nonvanishing.

    From \cite[Formula (7.4)]{baldwin2019lspace}, we have $\wti{\nu}_0\cdot \widebar{\Sigma}=1$ (the bundle set $\wti{\nu}_0$ was written as $\nu_0$ in \cite[Formula (7.2)]{baldwin2019lspace}). From \cite[Theorem 1.16 (5)]{baldwin2019lspace}, we have\[I^\sharp(W,\wti{\nu}_0\cup \widebar{\Sigma})=\sum_{i=-1}^1 I^\sharp(W,\wti{\nu}_0\cup \widebar{\Sigma};s_i)=\sum_{i=-1}^1 (-1)^{i+1}I^\sharp(W,\wti{\nu}_0;s_i),\]which cannot be equal to $I^\sharp(W,\wti{\nu}_0)$, even up to sign.

    On the other hand, from the long exact sequence of the pair $(W,\partial W)$, we have an isomorphism
	\[
		H_2(W,\partial W;\F_2) \xra{\cong} H_1(\partial W;\F_2) \cong H_1 (S^3_0(K);\F_2)\cong \F_2.
	\]
	Thus, we conclude that \[[\wti{\nu}_0] = [\wti{\nu}_0 \cup \widebar{\Sigma}]\in H_2(W, \partial W;\F_2)\]because they have the same boundary. The crucial point in this example is that the map\[H_2(W;\F_2)\to H_2(W,\partial W;\F_2)\]in the long exact sequence vanishes, though both are isomorphic to $\F_2$.
\end{example}

\section{Surgery exact triangle revisited}\label{sec: triangle}
In this section, we revisit the surgery exact triangle for instanton homology. We first focus on the general case and then specialize to the case of surgeries on knots in $S^3$. Again, we fix a coefficient ring $\cR$.

\bprop[{\cite[Theorem 2.1]{scaduto2015instantons}}]\label{prop: surgery exact triangle original}
Suppose $(Y,\omega,K)$ is a surgery tuple as in Definition \ref{defn: surgery tuple}. Recall that $\mu,\lambda\subset \partial (Y\bbslash K)$ are the meridian and the framed longitude of $K$, respectively. Then there exists an exact triangle
\begin{equation}\label{eq: surgery exact triangle, original}
	\xymatrix{
	I(Y,\omega;\cR)\ar[rr]^{F^\infty_0}&& I(Y_0(K),\omega\cup \mu;\cR)\ar[dl]^{F^0_1}\\
	&I(Y_1(K),\omega;\cR)\ar[lu]^{F^1_\infty}&
	}
\end{equation}
where the maps are cobordism maps associated to the surgery cobordisms with certain bundle sets. We will call the components in the bundle set other than $\omega$ and $\omega\times I$ the \emph{extra} bundle sets.
\eprop
Based on the discussion in \cite[\S 2.3]{alfieri2020framed}, we obtain the following result about the bundle sets.

\blem\label{lem: bundle set}
The bundle sets for the cobordism maps in \eqref{eq: surgery exact triangle, original} are described as follows.
\begin{itemize}
    \item The bundle set for $F^\infty_0$ consists of $\omega\times I$ and the cocore disk $\Dcc{\infty}{0}$.
    \item The bundle set for $F^0_1$ consists of $\omega \times I$ and the core disk $\Dc{0}{1}$.
    \item The bundle set for $F^1_\infty$ consists of $\omega \times I$.
\end{itemize}
\elem
\bpf
The first two cases follow directly from the discussion in \cite[\S 2.3]{alfieri2020framed}, though one needs to be careful about the issue in Remark \ref{rem: dependence wrong}. For the sake of completeness of the paper, we state the proof more explicitly as follows. 

The first case follows from \cite[\S 3.7]{scaduto2015instantons} directly.  Note that the meridian $\mu$ is isotopic to the dual knot $\wti{K}_0\subset Y_0(K)$, which is the boundary of the cocore disk in the first cobordism. 

In the second case, the bundle set is originally described as the union $(\omega\cup \wti{K}_0)\times I$ and the cocore disk $\Dcc{0}{1}$. Note that the boundary of the cocore disk is the dual knot $\wti{K}_1\subset Y_1(K)$, which is isotopic to $\lambda$ in $Y_1(K)$. Let $D_1\subset Y_0(K)$ be a properly embedded disk in the Dehn-filling solid torus. Note that the outgoing end of the bundle set is $\omega\cup \mu\cup \lambda$. To make the target of $F^0_1$ become $I(Y_1(K),\omega;\cR)$, we also need to add $D_1$ into the bundle set, which induces an isomorphism by Lemma \ref{lem: I_S is an iso}.

We also consider disks $D_\infty\subset Y=Y_\infty(K)$ and $D_0\subset Y_0(K)$ similarly. Then the cocore disk $\Dcc{0}{1}$ is isotopic to $(\lambda\times I)\cup D_0$ and the core disk $\Dc{0}{1}$ is isotopic to $((\mu+\lambda)\times I)\cup D_1\cup D_0$, where the $D_0$ in the cocore disk is to make its boundary become $\mu$ instead of $\mu+\lambda$. Hence the bundle set for $F^0_1$ is\[\begin{aligned}
    ((\omega\cup \mu)\times I)\cup \Dcc{0}{1}\cup D_1)=&((\omega\cup \mu\cup \lambda)\times I)\cup D_0\cup D_1\\=&((\omega\cup (\mu+ \lambda))\times I)\cup D_0\cup D_1\\=&(\omega\times I)\cup \Dc{0}{1}.
\end{aligned},\]where we implicitly use Lemma \ref{lem: cob map depending on H_2} in the second equation.

A similar trick involving the isotopies of the cocore disk and the core disk can be applied to the third case. The bundle set in the third case is originally described as the union $(\omega\cup (\mu+ \lambda))\times I$ and the cocore disk $\Dcc{1}{\infty}$, whose incoming end is $\omega\cup (\mu+\lambda)$ and outgoing end is $\omega\cup 2\cdot (\mu+\lambda)$. To make the source and the target of $F^1_{\infty}$ compatible with \eqref{eq: surgery exact triangle, original}, we need to add $D_1$ for the incoming end and add a trivial annulus $A$ with boundary $2\cdot(\mu+\lambda)$ for the outgoing end. The cocore disk is isotopic to $((\mu+ \lambda)\times I)\cup D_1$. Hence the bundle set for $F^1_\infty$ is\[\begin{aligned}
   ((\omega\cup (\mu+ \lambda))\times I)\cup \Dcc{1}{\infty}\cup D_1\cup A=&  (\omega\times I)\cup 2\cdot(\mu+\lambda)\times I\cup 2\cdot D_1\cup A.
\end{aligned}\]Note that $2\cdot(\mu+\lambda)\times I\cup 2\cdot D_1\cup A$ bounds an embedded $3$-sphere in the cobordism and hence represents the trivial mod $2$ homology class. By Lemma \ref{lem: cob map depending on H_2}, this union does not affect the cobordism map up to sign. Hence the bundle set for $F^1_\infty$ is just $\omega\times I$.
\epf
To specify the bundle sets in the cobordism maps, we introduce the following definition.
\bdefn\label{defn: bundle core cocore}
For $\ep=\ep_1\ep_2\in\{00,01,10,11\}$ and $|p_1q_2-p_2q_1|=1$, we write $F^{p_1/q_1}_{p_2/q_2}(\ep)$ for the cobordism map associated to the surgery cobordism from $Y_{p_1/q_1}(K)$ and $Y_{p_2/q_2}(K)$ with certain bundle sets, where $\ep_1$ and $\ep_2$ denote the presence of the core disk and the cocore disk, respectively. The product $\omega\times I$ will always be included in the bundle set.
\edefn
By varying the choice of $\omega$ in Proposition \ref{prop: surgery exact triangle original}, we can obtain exact triangles with the same vertices and different maps. We first consider the case when $(Y,\omega,K)$ is a nontrivial admissible surgery tuple. Note that adding any unoriented $1$-submanifold $\omega'\subset Y$ disjoint from $\Sigma$ to $\omega$ in Definition \ref{defn: surgery tuple} yields another nontrivial admissible surgery tuple $(Y,\omega\cup \omega',K)$. Since $\mu\cdot \lambda=-1$, we know that $\mu$ and $\lambda$ generate \begin{equation}\label{eq: H1}
    H_1(\partial (Y\bbslash K);\F_2)\cong \F_2\op \F_2.
\end{equation}Hence there are four natural choices of $\omega'$ that represent the four elements in \eqref{eq: H1}, namely\begin{equation}\label{eq: extra bundle}
    \omega'=0,\mu,\mu+\lambda,\lambda.
\end{equation}Hence we have the following proposition.
\bprop\label{prop: more triangles}
Suppose $(Y,\omega,K)$ is a nontrivial admissible surgery tuple as in Definition \ref{defn: surgery tuple}. Then we have the following exact triangles.
\begin{equation}\label{eq: surgery exact triangle, more 1}
	\xymatrix{
	I(Y,\omega;\cR)\ar[rr]^{F^\infty_0(01)}&& I(Y_0(K),\omega\cup \mu;\cR)\ar[dl]^{F^0_1(10)}\\
	&I(Y_1(K),\omega;\cR)\ar[lu]^{F^1_\infty(00)}&
	}
\end{equation}
\begin{equation}\label{eq: surgery exact triangle, more 2}
	\xymatrix{
	I(Y,\omega;\cR)\ar[rr]^{F^\infty_0(00)}&& I(Y_0(K),\omega;\cR)\ar[dl]^{F^0_1(01)}\\
	&I(Y_1(K),\omega\cup \mu;\cR)\ar[lu]^{F^1_\infty(10)}&
	}
\end{equation}
\begin{equation}\label{eq: surgery exact triangle, more 3}
	\xymatrix{
	I(Y,\omega\cup \lambda;\cR)\ar[rr]^{F^\infty_0(10)}&& I(Y_0(K),\omega;\cR)\ar[dl]^{F^0_1(00)}\\
	&I(Y_1(K),\omega;\cR)\ar[lu]^{F^1_\infty(01)}&
	}
\end{equation}
\begin{equation}\label{eq: surgery exact triangle, more 4}
	\xymatrix{
	I(Y,\omega\cup \lambda;\cR)\ar[rr]^{F^\infty_0(11)}&& I(Y_0(K),\omega\cup \mu;\cR)\ar[dl]^{F^0_1(11)}\\
	&I(Y_1(K),\omega\cup \mu;\cR)\ar[lu]^{F^1_\infty(11)}&
	}
\end{equation}
If $(Y,\omega,K)$ is a trivial admissible surgery tuple as in Definition \ref{defn: surgery tuple}, then we only have \eqref{eq: surgery exact triangle, more 1} and \eqref{eq: surgery exact triangle, more 4} as $(Y_0(K),\omega=0)$ is not admissible.
\eprop
\brem
For triangles in Proposition \ref{prop: more triangles}, the extra bundle set in each vertex other than $\omega$ is isotopic to the dual knot in the corresponding surgery manifold. The first three triangles appeared in \cite[Figure 1]{scaduto2015instantons}, where the second and the third are obtained from the first one by changing the nontrivial admissible surgery tuple. Here we obtain those two triangles by modifying $\omega$. The fourth triangle is new, but may be well-known to experts.
\erem
\bpf[Proof of Proposition \ref{prop: more triangles}]
We only consider the case of a nontrivial admissible surgery tuple. The case of a trivial admissible surgery tuple is similar. The first triangle follows directly from Proposition \ref{prop: surgery exact triangle original} and Lemma \ref{lem: bundle set}, with the notation from Definition \ref{defn: bundle core cocore}. The remaining three triangles are obtained from the first one by adding $\omega'$ in \eqref{eq: extra bundle} to $\omega$ and applying the identification of bundle sets in the proof of Lemma \ref{lem: bundle set}. More explicitly, we again suppose $D_0,D_1,D_\infty$ are disks in the Dehn filling solid tori of $Y_0(K),Y_1(K),Y=Y_\infty(K)$, respectively. Then we have the following identifications for core disks and cocore disks.\begin{equation}\label{eq: core and cocore}
    \begin{aligned}
        \Dc{\infty}{0}=&(\lambda\times I)\cup D_0,~\Dcc{\infty}{0}=(\mu\times I)\cup D_\infty,\\\Dc{0}{1}=&((\mu+\lambda)\times I)\cup D_1\cup D_0,~\Dcc{0}{1}=(\lambda\times I)\cup D_0\cup D_1,\\\Dc{1}{\infty}=&(\mu\times I)\cup D_\infty,\aand \Dcc{1}{\infty}=((\mu+\lambda)\times I)\cup D_1\cup D_\infty.
    \end{aligned}
\end{equation}where $D_0\subset \Dc{0}{1}$, $D_1\subset \Dcc{0}{1}$, and $D_\infty\subset \Dcc{1}{\infty}$ are added so that the ends of the bundle sets coincide with the dual knot representatives in \[\wti{K}_\infty=K=\lambda\subset Y=Y_\infty(K),~\wti{K}_0=\mu\subset Y_0(K),\aand \wti{K}_1=\mu\subset Y_1(K),\]Finally, we apply Lemmas \ref{lem: cob map depending on H_2} and \ref{lem: I_S is an iso} to obtain the results.
\epf

Using Lemmas \ref{lem: cob map depending on H_2} and \ref{lem: I_S is an iso}, we can further modify the exact triangles in Proposition \ref{prop: more triangles}. The ``half lives, half dies" theorem shows that \begin{equation}\label{eq: half live half die}
    \ker\left(H_1(\partial (Y\bbslash K;\F_2)\to H_1(Y\bbslash K;\F_2)\right)\cong \F_2.
\end{equation}Then we have the following.
\begin{itemize}
    \item $[K]=[\lambda]\neq 0\in H_1(Y;\F_2)$ if and only if $[\mu]$ generates \eqref{eq: half live half die};
    \item $[\wti{K}_0]=[\mu]\neq 0\in H_1(Y_0(K);\F_2)$ if and only if $[\lambda]$ generates \eqref{eq: half live half die};
    \item $[\wti{K}_1]=[\mu]\neq 0\in H_1(Y_1(K);\F_2)$ if and only if $[\mu+\lambda]$ generates \eqref{eq: half live half die}.
\end{itemize}In the nontrivial admissible surgery tuple case, we know that exactly one of the above cases will happen. In the trivial admissible surgery tuple case, we know that the second case happens. Without loss of generality, we assume that the second case happens. The following corollary follows directly from Lemmas \ref{lem: cob map depending on H_2} and \ref{lem: I_S is an iso}, and from homology computations.

\bcor\label{cor: more triangles 2}
Suppose $(Y,\omega,K)$ is a nontrivial admissible surgery tuple as in Definition \ref{defn: surgery tuple}. Suppose $[\lambda]$ generates \eqref{eq: half live half die} and hence also trivial in $H_1(Y;\F_2)$ and suppose $S\subset Y$ is a surface from Lemma \ref{lem: null-homologous curves bound embedded surfaces}. Suppose $D_1\subset Y_1(K)$ is the meridian disk and let $S+D_1$ be the surface obtained from $S\cup D_1$ by resolving intersection points. Then \eqref{eq: surgery exact triangle, more 2}, \eqref{eq: surgery exact triangle, more 3}, \eqref{eq: surgery exact triangle, more 4} reduce to the following triangles.
\begin{equation}\label{eq: surgery exact triangle, more 2 2}
	\xymatrix{
	I(Y,\omega;\cR)\ar[rr]^{F^\infty_0(00)}&& I(Y_0(K),\omega;\cR)\ar[dl]^{\quad\bI_{S+D_{1}}\circ F^0_1(01)}\\
	&I(Y_1(K),\omega;\cR)\ar[lu]^{F^1_\infty(10)\circ \bI_{S+D_{1}}'}&
	}
\end{equation}
\begin{equation}\label{eq: surgery exact triangle, more 2 3}
	\xymatrix{
	I(Y,\omega;\cR)\ar[rr]^{F^\infty_0(10)\circ \bI_{S}'}&& I(Y_0(K),\omega;\cR)\ar[dl]^{F^0_1(00)}\\
	&I(Y_1(K),\omega;\cR)\ar[lu]^{\bI_S\circ F^1_\infty(01)}&
	}
\end{equation}
\begin{equation}\label{eq: surgery exact triangle, more 2 4}
	\xymatrix{
	I(Y,\omega;\cR)\ar[rr]^{F^\infty_0(11)\circ \bI_{S}'}&& I(Y_0(K),\omega\cup \mu;\cR)\ar[dl]^{\quad\bI_{S+D_{1}}\circ F^0_1(11)}\\
	&I(Y_1(K),\omega;\cR)\ar[lu]^{\bI_S\circ F^1_\infty(11)\circ \bI_{S+D_{1}}'}&
	}
\end{equation}
Moreover, we have identifications of maps\begin{equation}\label{eq: identification cob}
    F^1_\infty(10)\circ \bI_{S+D_{1}}'=\pm \bI_S\circ F^1_\infty(01)\aand \bI_S\circ F^1_\infty(11)\circ \bI_{S+D_{1}}'=\pm F^1_{\infty}(00).
\end{equation}
If $(Y,\omega,K)$ is a trivial admissible surgery tuple as in Definition \ref{defn: surgery tuple}, then we only have \eqref{eq: surgery exact triangle, more 2 4} and the second equation in \eqref{eq: identification cob}.
\ecor
\brem\label{rem: other Floer}
One can compare triangles in Proposition \ref{prop: more triangles} and Corollary \ref{cor: more triangles 2} with the existing surgery exact triangles in Heegaard Floer theory and monopole Floer theory. The triangles in those theories are first established over $\F_2$ \cite{ozsvath2005double,kronheimer2007monopolesandlens}, where an essential ingredient is that a cobordism with an embedded sphere of self-intersection $-1$ induces a cobordism map with a multiple of $2$, which vanishes over $\F_2$. This is not true in general in instanton theory; see Remark \ref{rem: blow up}. This multiple of $2$ cannot be canceled by altering the homology orientation on the cobordism, and one must use twisted coefficients to establish the triangle when the characteristic of the coefficient ring is not $2$. Note that the twisted coefficient is roughly the analog of the extra bundle set. The case of Heegaard Floer theory was recently resolved by Abouzaid--Manolescu \cite[\S 6]{AM2025sign}, where they used a trick that is similar to \eqref{eq: surgery exact triangle, more 2 2} and \eqref{eq: surgery exact triangle, more 2 3}, i.e.\ pick the twisted coefficient cleverly so that the Floer homologies remain the same but two of the cobordism maps are modified and the compositions of two consecutive cobordism maps vanish. The case of monopole Floer theory was resolved over $\Q$ by Lin--Ruberman--Saveliev and over $\mathbb{Z}[i]$ by Freeman \cite{Freeman2021triangle}. Note that the one over $\Q$ is obtained by adding certain signs to the components of the cobordism maps, which might be induced from some twisted coefficients. The one over $\mathbb{Z}[i]$ is similar to \eqref{eq: surgery exact triangle, more 4}, where the twist coefficient is chosen for cycles from the dual knots and the union of the core and cocore disks. From \eqref{eq: surgery exact triangle, more 2 4}, at least one of the Floer homologies is not isomorphic to those without local coefficients.
\erem

\brem\label{rem: BS notation}
Other than \eqref{eq: identification cob}, there is no further relation for cobordism maps, even when the source and the target are the same, because the difference of the bundle set represents nontrivial mod $2$ homology class. Baldwin--Sivek \cite[\S 2.2]{baldwin2020concordance} used the trick of adding extra bundle sets to make all three vertices in the exact triangle have no extra bundle sets, which is indeed either \eqref{eq: surgery exact triangle, more 2 2} or \eqref{eq: surgery exact triangle, more 2 3}. More explicitly, from \cite[Equations (2.3) and (2.4)]{baldwin2020concordance}, for a knot $K\subset S^3$ and the cobordism $W^\infty_n:S^3\to S^3_n(K)$, we have the following result.

\begin{itemize}
    \item When $n$ is even, we use \eqref{eq: surgery exact triangle, more 2 2} and \eqref{eq: surgery exact triangle, more 1} with\[(Y,Y_0(K),Y_1(K))=(S^3,S_n^3(K),S_{n+1}^3(K))\]for the case where the extra bundle set is trivial and the meridian $\mu$, respectively.
    \item When $n$ is odd, we use \eqref{eq: surgery exact triangle, more 2 3} and \eqref{eq: surgery exact triangle, more 1} with \[(Y,Y_0(K),Y_1(K))=(S_n^3(K),S_{n+1}^3(K),S^3)\]for the case where the extra bundle set is trivial and the meridian (now it is $\mu+\lambda$ instead of $\mu$), respectively. 
\end{itemize}

Hence our notation of the bundle set in Definition \ref{defn: bundle core cocore} is related to Baldwin--Sivek's notation via\[I^\sharp(X_n,\nu_n)=F^{\infty}_{n}(00)\aand I^\sharp(X_n,\wti{\nu}_n)=F^{\infty}_{n}(01)\text{ when }n\text{ is even,}\]\[I^\sharp(X_n,\nu_n)=\bI_{S_n}\circ F^{\infty}_{n}(01)=\pm F^\infty_n(10)\circ \bI'_{S}\aand\]\[ I^\sharp(X_n,\wti{\nu}_n)=F^{\infty}_{n}(00)=\pm \bI_{S_n}\circ F^1_{\infty}(11)\circ \bI'_{S}\text{ when }n\text{ is odd;}\]where $S$ is a Seifert surface of $K$, $S_n$ is obtained from $S$, the meridian disk $D_n\subset S^3_n(K)$, and $|n-1/2|$ annuli with boundary $2\cdot \mu$ by resolving intersection points. Note that $\wti{\nu}_n$ is only defined for $n=0,-1$ in \cite{baldwin2022concordanceII} but it can be extended to any integer $n$ by considering the triangle obtained from adding the meridian ($\mu$ when $n$ is even and $\mu+\lambda$ when $n$ is odd) to the bundle set.
\erem

Finally, we consider surgery exact triangle for framed instanton homology of knot surgeries. We fix a knot $K\subset S^3$ and take its meridian $\mu$ and the Seifert longitude $\lambda$ with $\mu\cdot \lambda=-1$. We write\[S^3\bbslash K=S^3\backslash \mathrm{int}N(K)\]for the knot complement. 

For a surgery slope $r=p/q\in\mathbb{Q}\cup\{\infty=1/0\}$, we call it \emph{even} or \emph{odd} if the numerator $p$ is even or odd, respectively. 

For a fixed coefficient ring $\cR$, we write
\begin{equation}\label{eq: H notation}
    \begin{aligned}
    \cH(r,\omega)=&I^{\sharp}(S^3_{r}(K),\omega;\cR),\\
        \cH(r) = I^{\sharp}(S^3_{r}(K),\omega=0;\cR)&\aand\wti{\cH}(r) = I^{\sharp}(S^3_{r}(K),\omega=\wti{K}_r;\cR)
    \end{aligned}
\end{equation}
where the dual knot $\wti{K}_r$ is the bundle set rather than singular locus as in \cite[\S 4.3]{kronheimer2011khovanov}. If we want to specify the knot $K$, then we use $\cH(K,r,\omega)$, $\cH(K,r)$, and $\tcH(K,r)$ instead.

Note that the dual knot $\wti{K}_r\subset S^3_r(K)$ represents a nontrivial mod $2$ homology class if and only if $r$ is even. Due to \eqref{eq: homology of surgery} and Lemma \ref{lem: I_S is an iso}, we always have \[[\wti{K}_r]=[\mu]\in H_1(S^3_r(K);\F_2),\]and there exists a canonical isomorphism\begin{equation*}\label{eq: canonical iso for surgery}
    \tcH(r)\cong \cH(r,\mu),
\end{equation*}
In particular, when $r\in\Z$, we pick $\wti{K}_r=\mu$ and when $r=\infty$, we pick $\wti{K}_\infty=\lambda$.

Then we have the following proposition analogous to Proposition \ref{prop: more triangles} and Corollary \ref{cor: more triangles 2}.

\bprop[{\cite[\S 7.5]{scaduto2015instantons}}]\label{prop: surgery exact triangle}
For $i=0,1,2$, suppose $p_i$ and $q_i$ are co-prime integers and suppose $r_i=p_i/q_i$. If $p_iq_{i+1}-p_{i+1}q_i=1$ for all $i\in\Z/3$, we call $(r_0,r_1,r_2)$ a \emph{slope triad}. In such a case, there exist exact triangles
\begin{equation*}
	\xymatrix{
	\cH(r_0)\ar[rr]^{F^{r_0}_{r_1}(01)}&& \tcH(r_1)\ar[dl]^{F^{r_1}_{r_2}(10)}\\
	&\cH(r_2)\ar[lu]^{F^{r_2}_{r_0}(00)}&
	}
    \xymatrix{
	\cH(r_0)\ar[rr]^{F^{r_0}_{r_1}(00)}&& \cH(r_1)\ar[dl]^{F^{r_1}_{r_2}(01)}\\
	&\tcH(r_2)\ar[lu]^{F^{r_2}_{r_0}(10)}&
	}
\end{equation*}
\begin{equation*}
	\xymatrix{
	\tcH(r_0)\ar[rr]^{F^{r_0}_{r_1}(10)}&& \cH(r_1)\ar[dl]^{F^{r_1}_{r_2}(00)}\\
	&\cH(r_2)\ar[lu]^{F^{r_2}_{r_0}(01)}&
	}
    \xymatrix{
	\tcH(r_0)\ar[rr]^{F^{r_0}_{r_1}(11)}&& \tcH(r_1)\ar[dl]^{F^{r_1}_{r_2}(11)}\\
	&\tcH(r_2)\ar[lu]^{F^{r_2}_{r_0}(11)}&
	}
\end{equation*}
where the notation $F^{r_i}_{r_{i+1}}(\ep)$ is from Definition \ref{defn: bundle core cocore}. Furthermore, if $r_1$ is even, or equivalently $r_0$ and $r_2$ are odd, then we can replace all $\tcH(r_0)$ and $\tcH(r_2)$ by $\cH(r_0)$ and $\cH(r_2)$, respectively and there are identifications of maps\begin{equation}\label{eq: id of maps}
    F^{r_2}_{r_0}(10)=\pm F^{r_2}_{r_0}(01)\aand F^{r_2}_{r_0}(00)=\pm F^{r_2}_{r_0}(11).
\end{equation}
\eprop

\section{Embedded spheres in cobordisms}\label{sec: embedded sphere}
In this section, we review relations for cobordism maps when there exist embedded spheres of self-intersection $0,-1,-2$ and prove Theorem \ref{thm: distance two triangle CDX} as a byproduct. Again, we fix a coefficient ring $\cR$.

\blem\label{lem: embedded sphere}
Suppose $(W,\nu):(Y_0,\omega_0)\to (Y_1,\omega_1)$ is a cobordism between admissible pairs. Suppose $S\subset W$ is an embedded sphere with $S\cdot S=k$ and $|\nu\cap S|=l$. Then we have the following relations for the cobordism map\[I(W,\nu):I(Y_0,\omega_0;\cR)\to I(Y_1,\omega_1;\cR).\]
\begin{enumerate}
    \item\label{t1} If $k=0$ and $l$ is odd, then $I(W,\nu)=0$.
    \item\label{t2} If $k=-1$ and $l$ is odd, then $I(W,\nu)=0$.
    \item\label{t3} If $k=-1$ and $l=0$, and $(W,\nu)\cong (W',\nu')\#(\overline{\mathbb{CP}^2},0)$ for another cobordism $(W',\nu')$ between the same admissible pairs, then $I(W,\nu)=\pm I(W',\nu')$.
    \item\label{t4} If $k=-2$ and $l$ is odd, then $I(W,\nu)=\pm I(W,\nu\cup S)$. 
\end{enumerate}
\elem
\bpf
All terms follow from the neck-stretching argument and the analysis in the neighborhood of $S$. Term \eqref{t1} follows directly from the proof of \cite[Lemma 3.10]{LY2025torsion}. Term \eqref{t2} follows from \cite[\S 5.2]{scaduto2015instantons}, which is an important step in the proof of the surgery exact triangle. 

Term\eqref{t3} follows from the blow-up formula for the Donaldson invariant (for dimension-zero part of the moduli space); see \cite[Theorem (4.8)]{Donaldson1990polynomial}, \cite[Proposition (9.3.14)]{DK1992instanton}, \cite{Ozsvath1994blowup}, and \cite[p. 942--943]{Kronheimer1997obstruction} for the closed case and \cite[Proposition 5.2 (2)]{kronheimer2011knot} for the cobordism case. Note that it is not the blow-up formula for Donaldson series as in \cite[\S 5.1]{baldwin2019lspace}, which only works over $\C$. 

Term\eqref{t4} follows from the last paragraph in the proof of \cite[Proposition 5.11]{DMML2024distancetwo}; see also \cite[Equation (6.32)]{CDX2020polygon}.
\epf
\brem\label{rem: blow up}
Lemma \ref{lem: embedded sphere}\eqref{t2}\eqref{t3} over $\C$ also follows from \cite[Theorem 1.16 (4)]{baldwin2019lspace}, which provides a blow-up formula for the decomposition of the cobordism map analogous to spin$^c$ decomposition in Heegaard Floer and monopole Floer theories. The factor $1/2$ in \cite[Theorem 1.16 (4)]{baldwin2019lspace} makes the instanton theory different from the other two Floer theories; compare \cite[Theorem 1.4]{OS2006smooth} and \cite[Theorem 39.3.1 and Equation (39.6)]{kronheimer2007monopoles}.
\erem
In our previous work \cite[\S 3.3]{LY2025torsion}, we use Lemma \ref{lem: embedded sphere}\eqref{t1} for $\cR=\F_2$ to derive that compositions of maps from surgery exact triangles vanish. Indeed, this vanishing result also follows from Lemma \ref{lem: embedded sphere}\eqref{t2}\eqref{t3}. We describe the situation more explicitly as follows, since it provides some clue to Theorem \ref{thm: distance two triangle CDX} and we will use the idea to study rational surgeries in \S \ref{sec: rational surgeries}.

We start with the setup of Proposition \ref{prop: surgery exact triangle}, i.e.\ for $i=0,1,2$, suppose $p_i$ and $q_i$ are co-prime integers with\begin{equation}\label{eq: condition p q}
    p_0q_1-p_1q_0=p_1q_2-p_2q_1=p_2q_0-p_0q_2=1.
\end{equation}Note that $p_i/q_i=(-p_i)/(-q_i)$, but the equations in \eqref{eq: condition p q} no longer hold if we replace $(p_i,q_i)$ by $(-p_i,-q_i)$. We have\[(-p_2)q_1-(-q_2)p_1=1,\]and we can find another (unique) pair of co-prime integers $(p_3,q_3)$ such that\[p_3(-q_2)-(-p_2)q_3=p_1q_3-p_3q_1=1.\]Take $r_i=p_i/q_i$ for $i=0,1,2,3$ (e.g. $(r_0,r_1,r_2,r_3)=(-1,0,\infty,1)$), we have the following surgery triads of $3$-manifolds
\begin{equation}\label{eq: extended triad}
	\xymatrix{
		&S^3_{r_0}(K)\ar[dr]^{W^{r_0}_{r_1}}&\\
		S^3_{r_2}(K)\ar@<-2pt>[rr]_{W^{r_2}_{r_1}}\ar[ur]^{W^{r_2}_{r_0}} && S^3_{r_1}(K)\ar[dl]^{W^{r_1}_{r_3}}\ar@<-2pt>[ll]_{W^{r_1}_{r_2}}\\
		&S^3_{r_3}(K)\ar[lu]^{W^{r_3}_{r_2}}&
	}
\end{equation}such that we can apply the exact triangles in Proposition \ref{prop: surgery exact triangle} to either surgery triad. To obtain embedded spheres as in Lemma \ref{lem: embedded sphere}, we consider the compositions of two surgery cobordisms in \eqref{eq: extended triad} and take the union of the cocore disk in the first cobordism and the core disk in the second cobordism. We denote the embedded sphere by\begin{equation}
    S(a,b,c)=\Dcc{a}{b}\cup \Dc{b}{c}\subset W^b_c\circ W^a_b
\end{equation}for $W^b_c$ and $W^a_b$ in \eqref{eq: extended triad} and write $S(a,b,c)^2$ for its self-intersection number. Then we know from Kirby calculus (in particular, taking $(r_0,r_1,r_2,r_3)=(-1,0,\infty,1)$ and using the slam-dunk) that \begin{equation}\label{eq: self intersection}
\begin{aligned}
    0=&S(r_1,r_2,r_1)^2=S(r_2,r_1,r_2)^2\\
    -1=&S(r_0,r_1,r_2)^2=S(r_1,r_2,r_0)^2=S(r_2,r_0,r_1)^2\\
    -1=&S(r_2,r_1,r_3)^2=S(r_1,r_3,r_2)^2=S(r_3,r_2,r_1)^2\\
    -2=&S(r_0,r_1,r_3)^2=S(r_3,r_2,r_0)^2.
\end{aligned}
\end{equation}
Moreover, we have\begin{equation}\label{eq: composition blowup}
    W^{r_2}_{r_1}=(W^{r_0}_{r_1}\circ W^{r_2}_{r_0})\#\overline{\mathbb{CP}^2}\aand     W^{r_1}_{r_2}=(W^{r_3}_{r_2}\circ W^{r_1}_{r_3})\#\overline{\mathbb{CP}^2}
\end{equation}

If we apply the first triangle in Proposition \ref{prop: surgery exact triangle} to the top triad of \eqref{eq: extended triad} and the fourth triangle in Proposition \ref{prop: surgery exact triangle} to the bottom triad of \eqref{eq: extended triad}, then we obtain the following.
\begin{equation}\label{eq: two triangles}
    \xymatrix{
		\tcH(r_2)\ar[rr]^{F^{r_2}_{r_1}(11)} && \tcH(r_1)\ar[dl]^{F^{r_1}_{r_3}(11)}\ar[rr]^{F^{r_1}_{r_2}(10)}&& \cH(r_2)\ar[dl]^{F^{r_2}_{r_0}(00)}\\
		&\tcH(r_3)\ar[lu]^{F^{r_3}_{r_2}(11)}&&\cH(r_0)\ar[lu]^{F^{r_0}_{r_1}(01)}&
	}
\end{equation}
The main usage of Lemma \ref{lem: embedded sphere}\eqref{t1} in \cite[\S 3.3]{LY2025torsion} is the vanishing result of the composition\begin{equation}\label{eq: vanishing}
    F^{r_1}_{r_2}(10)\circ F^{r_2}_{r_1}(11)=0,
\end{equation}because $S(r_2,r_1,r_2)^2=0$ and the bundle set intersects $S(r_2,r_1,r_2)$ at one point (the intersection of the core disk and the cocore disk in $W^{r_2}_{r_1}$). Alternatively, we can derive the vanishing result from\begin{equation}\label{eq: alternative vanish}
    F^{r_1}_{r_2}(10)=\pm F^{r_3}_{r_2}(10)\circ F^{r_1}_{r_3}(11)\aand F^{r_1}_{r_3}(11)\circ F^{r_2}_{r_1}(11)=0,
\end{equation}where the first equation is from Lemma \ref{lem: embedded sphere}\eqref{t3} and \eqref{eq: composition blowup} and the second equation follows from Lemma \ref{lem: embedded sphere}\eqref{t2} or the exactness of the triangle. We perform careful computation of the bundle sets for the cobordisms in the first equation of \eqref{eq: alternative vanish} as follows.

\blem\label{lem: blow-up many cases}
Suppose $(r_1,r_3,r_2)$ is a slope triad as in Proposition \ref{prop: surgery exact triangle}. Then we have\[F^{r_1}_{r_2}(00)=\pm F^{r_3}_{r_2}(00)\circ F^{r_1}_{r_3}(00),F^{r_1}_{r_2}(10)=\pm F^{r_3}_{r_2}(10)\circ F^{r_1}_{r_3}(11)\]\[F^{r_1}_{r_2}(01)=\pm F^{r_3}_{r_2}(11)\circ F^{r_1}_{r_3}(01)\aand F^{r_1}_{r_2}(11)=\pm F^{r_3}_{r_2}(01)\circ F^{r_1}_{r_3}(10)\]
\elem
\bpf
The first equation follows directly from Lemma \ref{lem: embedded sphere}\eqref{t3} and \eqref{eq: self intersection}. The rest equations follow similarly, but need more computations on the bundle sets. We do the computation for the second equation carefully as follows. The computations for the last two equations are similar and we omit them.

Up to a change of basis for $\partial (S^3\bbslash K)$, we can assume $(r_1,r_3,r_2)=(0,1,\infty)$ and apply the identifications in \eqref{eq: core and cocore}, that is\[\begin{aligned}
    \Dc{r_1}{r_2}=& (\mu \times [0,2])\cup D_\infty,~\Dc{r_3}{r_2}=(\mu \times [1,2])\cup D_\infty,\\\Dc{r_1}{r_3}=&((\mu+\lambda) \times [0,1])\cup D_1\cup D_0,\aand \Dcc{r_1}{r_3}=(\lambda \times [0,1])\cup D_0\cup D_1,
\end{aligned}\]where we replace $I$ by actual intervals to indicate that it is in different cobordisms. Hence the bundle set in $F^{r_3}_{r_2}(10)\circ F^{r_1}_{r_3}(11)$ is \[\Dc{r_3}{r_2}\cup \Dc{r_1}{r_3}\cup \Dcc{r_1}{r_3}=\mu\times [0,2]\cup D_\infty\cup 2\cdot (\lambda\times [0,1])\cup 2\cdot D_0\cup 2\cdot D_1.\]The pieces with multiple $2$ can be removed by Lemma \ref{lem: cob map depending on H_2}, and the union of the rest pieces is disjoint from $S(r_1,r_3,r_2)$ and becomes $\Dc{r_1}{r_2}$ after blow-down.
\epf

By \eqref{eq: two triangles} and \eqref{eq: vanishing} with $(r_0,r_1,r_2,r_3)=(-1,0,\infty,1)$, together with the octahedral lemma \cite[Lemma A.3.10]{OSS2015grid}, we obtain the following exact triangle\begin{equation*}
	\xymatrix{
	\cH(-1)\ar[rr]^{F^{0}_{1}(11)\circ F^{-1}_{0}(01)}&& \tcH(1)\ar[dl]^{f}\\
	&\cH(\infty)\oplus \tcH(\infty)\ar[lu]^{g}&
	}
\end{equation*}
for some \emph{abstract} maps $f$ and $g$, which means those maps are not necessarily cobordism maps. Replacing Proposition \ref{prop: surgery exact triangle} with Proposition \ref{prop: more triangles}, we obtain the exact triangle in Theorem \ref{thm: distance two triangle CDX} except the explicit descriptions of $f$ and $g$. Note that the proof also works for a trivial admissible surgery pair, as we only use the first and the fourth exact triangles. This is the simpler proof mentioned in the \ref{rem: octahedral trick}.

To prove Theorem \ref{thm: distance two triangle CDX} with all maps identified as cobordism maps, we need to further use Lemma \ref{lem: embedded sphere}\eqref{t4} and diagram chasing as in the proof of \cite[Proposition 5.3]{LY2022integral1}.
\bthm\label{thm: distance two}
Suppose $(Y,\omega,K)$ is a surgery tuple as in Definition \ref{defn: surgery tuple}. Then there exists an exact triangle\begin{equation*}\label{eq: surgery exact triangle, revisited}
	\xymatrix{
	I(Y_{-1}(K),\omega;\cR)\ar[rr]^{F^{0}_{1}(11)\circ F^{-1}_{0}(01)}&& I(Y_{1}(K),\omega\cup \mu;\cR)\ar[dl]^{\quad\quad \left(F^1_\infty(10),F^1_\infty(11)\right)}\\
	&I(Y,\omega;\cR)\oplus I(Y,\omega\cup \lambda;\cR)\ar[lu]^{F^\infty_{-1}(00)+F^\infty_{-1}(10)\quad\quad}&
	},
\end{equation*}where the maps are from Definition \ref{defn: bundle core cocore}, with suitable choice of signs.
\ethm
\bpf
Replacing Proposition \ref{prop: surgery exact triangle} with Proposition \ref{prop: more triangles} in \eqref{eq: two triangles} and taking \[(r_0,r_1,r_2,r_3)=(-1,0,\infty,1),\] we obtain the following two triangles (we omit $\cR$). Recall that when $r\in\Z$, we pick $\wti{K}_r=\mu$ and when $r=\infty$, we pick $\wti{K}_\infty=\lambda$.

\begin{equation}\label{eq: two triangles 2}
    \xymatrix{
		I(Y,\omega\cup \lambda)\ar[rr]^{F^{\infty}_{0}(11)} && I(Y_0(K),\omega\cup \mu)\ar[dl]^{F^{0}_{1}(11)}\ar[rr]^{F^{0}_{\infty}(10)}&& I(Y,\omega)\ar[dl]^{F^{\infty}_{-1}(00)}\\
		&I(Y_1(K),\omega\cup \mu)\ar[lu]^{F^{1}_{\infty}(11)}&&I(Y_{-1}(K),\omega)\ar[lu]^{F^{-1}_{0}(01)}&
	}
\end{equation}

From Lemma \ref{lem: embedded sphere}\eqref{t4} and \eqref{eq: self intersection}, we have\[F^\infty_{-1}(00)\circ F^1_\infty(10)=\pm F^\infty_{-1}(10)\circ F^1_\infty(11).\]By suitably choosing signs, we assume that\begin{equation}\label{eq: -2 application}
    F^\infty_{-1}(00)\circ F^1_\infty(10)+F^\infty_{-1}(10)\circ F^1_\infty(11)=0.
\end{equation}Then we use diagram chasing to prove the exactness at all three vertices. We will use the exactness in \eqref{eq: two triangles 2} without mentioning it.

First, we consider the vertex $I(Y,\omega)\oplus I(Y,\omega\cup \lambda)$. Suppose\[(x,y)\in \ker\left(F^\infty_{-1}(00)+F^\infty_{-1}(10)\right),\]\[\text{i.e.}~F^\infty_{-1}(00)(x)+F^\infty_{-1}(10)(y)=0.\]We prove that\begin{equation}\label{eq: result 1}
    (x,y)\in \img \left(F^1_\infty(10),F^1_\infty(11)\right).
\end{equation}Indeed, we have \[\begin{aligned}
    0=&F^{-1}_0(01)(0)\\
    =&F^{-1}_0(01)\circ \left(F^\infty_{-1}(00)(x)+F^\infty_{-1}(10)(y)\right)\\
    =&F^{-1}_0(01)\circ F^\infty_{-1}(10)(y)\\
    =&\pm F^\infty_{0}(11)(y),
\end{aligned}\]where the last equation is from Lemma \ref{lem: blow-up many cases}. Then there exists $z\in I(Y_1(K),\omega\cup \mu)$ such that \[F^{1}_{\infty}(11)(z)=y.\]Then we have\[\begin{aligned}
    F^\infty_{-1}(00)\left(x-F^1_\infty(10)(z)\right)=&F^\infty_{-1}(00)(x)+F^\infty_{-1}(10)\circ F^1_\infty(11)(z)\\=&F^\infty_{-1}(00)(x)+F^\infty_{-1}(10)(y)\\=&0,
\end{aligned}\]where the first equation is from \eqref{eq: -2 application}. Then there exists $w\in I(Y_0(K),\omega\cup \mu)$ such that\[F^0_{\infty}(10)(w)=x-F^1_\infty(10)(z).\]Then we have\[F^{1}_{\infty}(11)\left(z\pm  F^0_{1}(11)(w)\right)=F^{1}_{\infty}(11)(z)=y,\aand\]\[
    F^1_\infty(10)\left(z+ F^0_{1}(11)(w)\right)=F^1_\infty(10)(z)\pm F^0_\infty(10)(w),\]where the last equation is from Lemma \ref{lem: blow-up many cases}. If the sign is positive, then the existence of $z+ F^0_{1}(11)(w)$ verifies \eqref{eq: result 1}. If the sign is negative, then the existence of $z- F^0_{1}(11)(w)$ verifies \eqref{eq: result 1}, which concludes the proof of the exactness at $I(Y,\omega)\oplus I(Y,\omega\cup \lambda)$.

Second, we consider the exactness at $I(Y_{-1}(K),\omega)$. We will use Corollary \ref{cor: more triangles 2} and Lemma \ref{lem: blow-up many cases} freely. We have\[F^0_1(11)\circ F^{-1}_0(01)\circ F^\infty_{-1}(00)=F^0_1(11)(0)=0,\aand\]
\[F^0_1(11)\circ F^{-1}_0(01)\circ F^\infty_{-1}(10)=\pm F^0_1(11)\circ F^0_1(11)=0.\]Suppose\[u\in \ker\left(F^0_1(11)\circ F^{-1}_0(01)\right).\]Then there exists $y\in I(Y,\omega\cup\lambda)$ such that\[F^\infty_0(11)(y)=F^{-1}_0(01)(u).\]Then we have\[F^{-1}_0(01)\left(u+F^\infty_{-1}(10)(y)\right)=F^{-1}_0(01)(u)\pm F^\infty_0(11)(y)=0.\]Hence there exists $x\in I(Y,\omega)$ such that\[F^\infty_{-1}(00)(x)=u+F^\infty_{-1}(10)(y)~\mathrm{or}~u-F^\infty_{-1}(10)(y),\]which implies\[u\in \img\left(F^\infty_{-1}(00)+F^\infty_{-1}(10)\right)\]and concludes the proof.

Finally, we consider the exactness at $I(Y_1(K),\omega\cup\mu)$. We have\[F^1_\infty(10)\circ F^0_1(11)\circ F^{-1}_0(01)=\pm F^0_\infty(10)\circ F^{-1}_0(01)(0)=0,\aand\]
\[F^1_\infty(11)\circ F^0_1(11)\circ F^{-1}_0(01)=0\circ F^{-1}_0(01)=0.\]Suppose\[v\in \ker\left (F^1_\infty(10),F^1_\infty(11)\right).\]Then there exists $w\in I(Y_0(K),\omega\cup \mu)$ such that \[F^0_1(11)(w)=v.\]Moreover, we have\[F^0_\infty(10)(w)=\pm F^1_\infty(10)\circ F^0_1(11)(w)=\pm F^1_\infty(10)(v)=0.\]Then there exists $u\in I(Y_{-1},\omega)$ such that\[F^{-1}_0(01)(u)=w\]and hence\[v\in \img\left (F^0_1(11)\circ F^{-1}_0(01)\right),\]which concludes the proof.
\epf
\section{Integer surgeries}\label{sec: integer surgeries}
From now on, we consider a coefficient field $\bK$ instead of a general coefficient ring $\cR$. In \cite[Proposition 1.1]{LY2025torsion}, we showed that for $\bK=\F_2$, the sequence $\{\dim \cH(n)\}_{n\in\Z}$ is either V-shaped, W-shaped, or generalized W-shaped. In this section, we first recap this result in any field $\bK$, then eliminate the case of being generalized W-shaped, and finally study the cases of V-shaped and W-shaped more carefully. Following \eqref{eq: H notation}, we fix a knot $K\subset S^3$ and a field $\bK$, and write\[\begin{aligned}
\cH(r,\omega)&=\cH(K,r,\omega)=I^\sharp(S^3_r(K),\omega;\bK),\\
    \cH(r)&=\cH(K,r)=I^\sharp(S^3_r(K),\omega=0;\bK)\\\aand \tcH(r)&=\tcH(K,r)=I^\sharp(S^3_r(K),\omega=\wti{K}_r;\bK)\cong I^\sharp(S^3_r(K),\mu;\bK).
\end{aligned}\]Since we only consider the case $r=n\in\Z$ in this section and $\wti{K}_n=\mu$, we use $\cH(r,\mu)$ instead of $\tcH(r)$.

The main result of this section is the following.

\bprop\label{prop: dim formula for integers}
	Suppose $K\subset S^3$ is a knot and $\bK$ is a field. There exists a concordance invariant $\nk{K}\in\Z$ satisfying $\nk{\widebar{K}}=\nk{K}$ for the mirror knot $\widebar{K}$. Moreover, for \begin{equation*}
	    r_{\bK}(K)=\min \left\{\dim \cH(\nk{K}),\dim \cH(\nk{K},\mu)\right\},
	\end{equation*}we have
	\begin{equation}\label{eq: dim for integer}
	    \dim \cH(n) = \dim \cH(n,\mu) = r_{\bK}(K) + |n-\nk{K}|
	\end{equation}
	for any integer $n\neq \nk{K}$. Furthermore, if $\nk{K}$ is odd, then \eqref{eq: dim for integer} also holds for $n=\nk{K}$. If $\nk{K}$ is even, then we have
	\begin{equation*}\label{eq: pm 2}
	    \{\dim \cH(\nk{K}),\dim \cH(\nk{K},\mu)\}=\{r_{\bK}(K),r_{\bK}(K)+2\}.
	\end{equation*}
\eprop
\bpf
This is re-statement of Propositions \ref{prop: concordance invariants}, \ref{prop: shape of the dimension sequence}, \ref{prop: differed at most by 2}, and \ref{prop: V and W}, with\[
	\nu^{\sharp}_{\bK}(K) = \frac{1}{2}(\nkp{K} + \nkm{K}).
\]
\epf

\subsection{Three kinds of shapes}
We first sketch the analogous results in the proof of \cite[Proposition 1.1]{LY2025torsion} for an arbitrary field $\bK$.

\begin{lemma}\label{lem: differ by 1}
	For any integer $n$, we have 
	\[
		\dim \cH(n+1,\omega) = \dim \cH(n,\omega') \pm 1.
	\]
	for arbitrary choices of bundle sets $\omega$ and $\omega'$.
\end{lemma}
\begin{proof}
	This follows directly from \eqref{eq: homology of surgery}, Proposition \ref{prop: more triangles}, and the fact that
	\[
		\dim I^{\sharp}(S^3;\bK) = 1.
	\]
\end{proof}

% \begin{lemma}\label{lem: oscillate when (n) neq (n,mu)}
% 	Suppose for a (necessarily even) integer $n$, we have
% 	\[
% 		\dim \cH(n) \neq \dim \cH(n,\mu)
% 	\]
% 	Then necessarily
% 	\[
% 		|\dim \cH(n) - \dim \cH(n,\mu)| = 2,
% 	\]
% 	and
% 	\[
% 		\dim \cH(n-1) = \dim \cH(n+1) = \frac{1}{2}(\dim \cH(n)+\dim \cH(n,\mu)).
% 	\]
% \end{lemma}
% \begin{proof}
% 	This is essentially \cite[Lemma 3.11]{LY2025torsion} and is a straightforward application of Corollary \ref{cor: all surgery exact triangles} and Lemma \ref{lem: differ by 1}.
% \end{proof}

\begin{lemma}\label{lem: monotone for large enough n}
	There exists an integer $N>0$ such that for any fixed choice of the bundle set $\omega$, we have the following results.
	\begin{itemize}
		\item For any integer $n>N$, we have
		\[
			\dim \cH(n+1,\omega) = \dim \cH(n,\omega) + 1
		\]
		\item For any integer $n<-N$, we have 
		\[
			\dim \cH(n-1,\omega) = \dim \cH(n,\omega) + 1
		\]
	\end{itemize} 
\end{lemma}
\bpf
\cite[Theorem 1.1]{baldwin2020concordance} and \cite[Theorem 1.12]{baldwin2022concordanceII} imply that when $\bK = \mathbb{C}$, such a monotonicity condition holds for $|n|>N_{\mathbb{C}}$ for some sufficiently large fixed integer $N_{\mathbb{C}}$. From the universal coefficient theorem, we know that the difference
\[
\dim I^{\sharp}(S^3_n(K),\omega;\bK) - \dim_{\mathbb{C}} I^{\sharp}(S^3_n(K),\omega;\mathbb{C})
\]
is a finite non-negative integer, which is non-increasing as $n$ increases and $n > N_{\mathbb{C}}$ (resp. as $n$ decreases and $n< -N_{\mathbb{C}}$) by Lemma \ref{lem: differ by 1}. Thus, it will eventually stabilize, which concludes the proof of the lemma.
\epf

% \begin{lemma}
% 	Suppose for an even integer $n$, we have
% 	\[
% 		\dim \cH(n) = \dim \cH(n,\mu)
% 	\]
% 	Then the following holds.
% 	\begin{itemize}
% 		\item If we have
% 		\[
% 			\dim \cH(n+1) = \dim \cH(n) + 1,
% 		\]
% 		then
% 		\[
% 			\dim \cH(n+2) = \dim \cH(n+2,\mu) = \dim \cH(n) + 2.
% 		\]
% 		\item If we have
% 		\[
% 			\dim \cH(n-1) = \dim \cH(n) + 1,
% 		\]
% 		then
% 		\[
% 			\dim \cH(n-2) = \dim \cH(n+2,\mu) = \dim \cH(n) + 2.
% 		\]
% 	\end{itemize}
% \end{lemma}
% \begin{proof}
% 	The proofs of \cite[Lemmas 3.12-3.15]{LY2025torsion} apply verbatim. The crucial ingredient is Lemma \ref{lem: composition zero}.
% \end{proof}

\begin{definition}
Suppose $K\subset S^3$ is a knot. Define
\[
\nkp{K} = \min\{n~|~\forall k\ge n,~\dim \cH(k+1) = \dim \cH(k)+1\},
\]
\[
\aand \nkm{K} = \max\{n~|~\forall k\le n,~ \dim \cH(k-1) =\dim \cH(k)+1\}.
\]Note that $\nu^{\sharp,\bK}_{\pm}(\widebar{K})=-\nu^{\sharp,\bK}_{\mp}(K)$ for the mirror knot $\widebar{K}$ because $\dim \cH(\widebar{K},n)=\dim \cH(K,-n)$.
\end{definition}

\begin{proposition}\label{prop: concordance invariants}
	The invariants $\nu^{\sharp,\bK}_{\pm}(K)$ are concordance invariants.
\end{proposition}
\bpf
This is verbatim from the proof of \cite[Proposition 1.12]{LY2025torsion}.
\epf

\begin{proposition}\label{prop: shape of the dimension sequence}
For a knot $K\subset S^3$, we have the following two results regarding $\nu^{\sharp,\bK}_{\pm}(K)$.
\begin{itemize}
		\item For any $n\ge\nkp{K}$, we have
		\begin{equation}\label{eq: dst increasing}
			\dim \cH(n+1) = \dim \cH(n+1,\mu) = \dim \cH(n)+1 = \dim \cH(n,\mu)+1
		\end{equation}
		
		\item For any $n\le \nkm{K}$, we have
		\begin{equation}\label{eq: dst decreasing}
			\dim \cH(n-1) = \dim \cH(n-1,\mu) = \dim \cH(n)+1 = \dim \cH(n,\mu)+1
		\end{equation}
\end{itemize}
Furthermore, the sequence $\{\dim \cH(n)\}_{n\in\mathbb{Z}}$ has one of the following three shapes.
\begin{enumerate}
	\item V-shaped: we have that $\{\dim \cH(n)\}_{n\in\mathbb{Z}}$ is unimodal, i.e.\, it has a unique minimum at integer $m=\nkp{K}= \nkm{K}$. Furthermore, we have
	\[\dim \cH(n,\mu) = \dim \cH(n)\text{ for }n\neq m\aand \]\[\begin{cases}
	    \dim \cH(m,\mu) - \dim \cH(m)\in\{0,2\} &\text{ if }m\text{ is even};\\
     \dim \cH(m,\mu) = \dim \cH(m)&\text{ if }m\text{ is odd}.
	\end{cases}\] 

	\item W-shaped: we have $\nkp{K} = \nkm{K}+2$ and for $m = \nkp{K} - 1 = \nkm{K} + 1$, we have the following:
		\begin{itemize}
			\item If $m$ is even and $n\neq m$, then
			\[
				\dim \cH(m) = \dim \cH(m,\mu) + 2 \aand \dim \cH(n,\mu) = \dim \cH(n)
			\]
			\item If $m$ is odd and $n\neq m\pm 1$, then
			\[
				\dim \cH(m\pm 1) = \dim \cH(m\pm 1,\mu) - 2\aand \dim \cH(n,\mu) = \dim \cH(n)
			\]
		\end{itemize}

	\item Generalized W-shaped: we have $m=\nkp{K} > \nkm{K}+2$ and the following holds. For an even integer $n$ such that $n\in[\nkm{K},\nkp{K}]$, we have
		\[
		\begin{cases}
	    \dim \cH(n) = \dim \cH(n,\mu) + 2 &\text{ if } m \text{ is odd};\\	
            \dim \cH(n) = \dim \cH(n,\mu) - 2 &\text{ if } m \text{ is even},
		\end{cases}
		\]and for other integers $n$, we have \[\dim \cH(n) = \dim \cH(n,\mu).\]
\end{enumerate}
\end{proposition}

\begin{proof}
	The proof is verbatim from that of \cite[Proposition 3.17]{LY2025torsion}, by replacing \cite[Lemma 3.10]{LY2025torsion} with Lemma \ref{lem: embedded sphere}\eqref{t1} for $\cR=\bK$.
\end{proof}
\brem\label{rem: same parity}
From the illustration of Proposition \ref{prop: shape of the dimension sequence} in \cite[Fig. 7-8]{LY2025torsion}, we have \[\nkp{K} - \nkm{K}\in 2 \Z.\]
\erem

\subsection{Elimination of the third case}
Next, we eliminate the case of being generalized W-shaped. We adopt the notation in Proposition \ref{prop: surgery exact triangle} and Definition \ref{defn: bundle core cocore}. Since different knots will be used, we add the knot into the notation:\[\cH(K,r,\omega),~ W^{r_i}_{r_{i+1}}(K)\aand F^{r_{i}}_{r_{i+1}}(K,\ep).\]

From the proof of \cite[Lemma 5.2]{baldwin2020concordance}, for any $a,b\in\Z$ and any two knots $K_1,K_2\subset S^3$, there exists an embedding
\begin{equation}\label{eq: embedding, no bundle}
	W^{\infty}_{a+b}(K_1\# K_2) \hookrightarrow W^{\infty}_{a}(K_1) \natural W^{\infty}_{b}(K_2)
\end{equation}
Moreover, as in the proofs of \cite[Lemma 5.2]{baldwin2020concordance} and \cite[Lemma 8.1 and Proposition 8.2]{baldwin2022concordanceII}, we can consider bundle sets in \ref{eq: embedding, no bundle} and obtain an embedding
\begin{equation}\label{eq: embedding, with bundle}
	\left(W^{\infty}_{a+b}(K_1\# K_2),\nu_{K_1\# K_2}\right) \hookrightarrow \left(W^{\infty}_{a}(K_1),\nu_{K_1}\right ) \natural \left(W^{\infty}_{b}(K_2),\nu_{K_2}\right),
\end{equation}where \begin{equation}\label{eq: possible bundles}
    (\nu_{K_1},\nu_{K_2},\nu_{K_1\# K_2})\in \left\{(\nu_a,\nu_b,\nu_{a+b}),(\nu_0,\wti{\nu}_0,\wti{\nu}_{0}),(\wti{\nu}_0,\nu_0,\wti{\nu}_{0}),(\wti{\nu}_0,\wti{\nu}_0,\nu_{0})\right\}
\end{equation}for the notation from Remark \ref{rem: BS notation}. We prefer to use the notation in Definition \ref{defn: bundle core cocore} because it simplifies the computation of the bundle sets via the identifications \eqref{eq: core and cocore}. More precisely, we obtain the following lemma.

\blem\label{lem: cobordism injective}
Suppose \begin{equation}\label{eq: bundle cases}
    \ep,\ep',\ep''\in \{(00,00,00),(00,01,01),(01,00,01),(01,01,00),(10,10,10),(11,11,10)\}
\end{equation}If $F^{\infty}_a(K_1,\ep)$ and $F^{\infty}_b(K_2,\ep')$ are both injective (or equivalently, non-vanishing), so it $F^{\infty}_{a+b}(K_1\# K_2,\ep'')$. 
\elem
\brem
From Remark \ref{rem: BS notation}, the last three cases in \eqref{eq: possible bundles} follow from the second, third, and fourth cases in \eqref{eq: bundle cases} for $a=b=0$, respectively. The first case in \eqref{eq: possible bundles} follows from the first four cases in \eqref{eq: bundle cases}, depending on whether $a$ and $b$ are even or odd.
\erem
\bpf[Proof of Lemma \ref{lem: cobordism injective}]
It suffices to compute the bundle sets so that \eqref{eq: embedding, with bundle} holds. Suppose $\mu_K$ and $\lambda_K$ are the meridian and the Seifert longitude of $K$. Suppose $S_{K}$ is the Seifert surface of $K$. Suppose $D_n(K)\subset S_n^3(K)$ and $D_\infty(K)\subset S_\infty^3(K)=S^3$ are the meridian disks and suppose $\Dc{\infty}{n}(K)$ and $\Dcc{\infty}{n}(K)$ are the core and cocore disks in $W^\infty_n(K)$. From \eqref{eq: core and cocore}, we have\[\Dc{\infty}{n}(K)=(n\cdot \mu_K+\lambda_K)\times I \cup D_{n}(K)\cup n\cdot D_\infty(K)\]\[\aand \Dcc{\infty}{n}(K)=\mu_K\times I\cup D_\infty(K).\]Moreover, we have\[\mu_{K_1}\times I\simeq \mu_{K_2}\times I\simeq\mu_{K_1\# K_2}\times I,\]\[ (\lambda_{K_1}\times I)\cup  (\lambda_{K_2}\times I)\simeq \lambda_{K_1\# K_2}\times I,\]\[S_{K_1}\cup S_{K_2}\simeq S_{K_1\# K_2}\]under the embedding \eqref{eq: embedding, no bundle}, where $\simeq$ means mod $2$ homologous and we apply Lemma \ref{lem: cob map depending on H_2}. Hence we obtain the desired result by direct computations.
\epf

\bprop\label{prop: differed at most by 2}
	For any field $\bK$ and any knot $K\subset S^3$, we have 
	\[
		\nkp{K} - \nkm{K} \leq 2.
	\]
	Furthermore, if $\nkp{K} - \nkm{K} = 2$ then they must both be odd.
\eprop
\bpf
We will apply Lemma \ref{lem: cobordism injective} for $K_1=\widebar{K}$ and $K_2=K$. Note that their connected sum is slice, which is concordant to the unknot $U$. The unknot is W-shaped over any field $\bK$ with $\nu^{\sharp,\bK}_{\pm }=\pm 1$ since \[\dim I^\sharp(L(p,q);\bK)=|p|\aand \dim I^\sharp(S^1\times S^2;\bK)-2=\dim I^\sharp(S^1\times S^2,\mu;\bK)=0\]from \cite[Corollary 1.2 and \S 7.6]{scaduto2015instantons} and Lemma \ref{lem: embedded sphere}\eqref{t2}. Thus, from Proposition \ref{prop: concordance invariants}, $\widebar{K}\# K$ is also W-shaped over any field $\bK$. Moreover, we have \begin{equation}\label{eq: W-shaped slice}
    \begin{aligned}
        F^\infty_n(\widebar{K}\# K,10)&\neq 0\text{ for integer }n\in(-\infty,-2]\cap \{0\}\\\aand F^\infty_n(\widebar{K}\# K,00)&\neq 0 \text{ for integer }n\le 0.
    \end{aligned}
\end{equation}

From Remark \ref{rem: same parity}, we know that $\nu^{\sharp,\bK}_{\pm}(K)$ have the same parity. Then we consider the following two cases.

\noindent{\bf Case 1}. $\nu^{\sharp,\bK}_{\pm}(K)$ are both odd.  We take 
\[
	a = \nkp{K_1} - 1 = -\nkm{K} - 1 \aand b = \nkp{K_2} - 1.
\]
Then \[\dim \cH(K_1,a)=\dim \cH(K_1,a+1)+1,\]\[\aand \dim \cH(K_2,b)=\dim \cH(K_2,b+1)+1\]by definition of $\nkp{K}$. Then Corollary \ref{cor: more triangles 2} and Remark \ref{rem: BS notation} imply that
\[
F^{\infty}_a(K_1,10) \neq 0 \aand F^{\infty}_b(K_2,10)\neq 0
\]
Then Lemma \ref{lem: cobordism injective} implies that
\[
F^{\infty}_{a+b}(K \# \widebar{K},10) \neq 0.
\]
Thus, we conclude from \eqref{eq: W-shaped slice} that $a+b\leq 0$, which is equivalent to
\[
\nkp{K} - \nkm{K}\leq 2.
\]

\noindent{\bf Case 2}. $\nu^{\sharp,\bK}_{\pm}(K)$ are both even. We take 
\[
	a = \nkp{K_1} = -\nkm{K} \aand b = \nkp{K_2} - 1.
\]
% The subtlety here is that we should adopt the following exact triangles
% \[
% 	\xymatrix{
% 	\cH(K_1, \infty,\mu)=\cH(K_1, \infty)\ar[rr]^{F^{\infty}_{a}(K_1,\mu)}&& \cH(K_1,a,\mu)\ar[dl]^{\quad\quad\quad F^{a}_{a+1}(K_1,\mu\cup\nu^{cc}_{a})}\\
% 	&\cH(K_1,a+1,\mu)=\cH(K_1,a+1)\ar[lu]^{F^{a+1}_{\infty}(K_1,\mu\cup \nu^{c}_{a+1})\quad\quad\quad}&
% 	}
% \]
% \[
% 	\xymatrix{
% 	\cH(K_2, \infty,\mu)\ar[rr]^{F^{\infty}_{b}(K_2,\mu)}&& \cH(K_2,b,\mu)=\cH(K_2,b)\ar[dl]^{\quad\quad\quad F^{b}_{b+1}(K_2,\mu\cup\nu^{cc}_{b})}\\
% 	&\cH(K_1,b+1,\mu\cup\mu)=\cH(K_2,b+1)\ar[lu]^{F^{b+1}_{\infty}(K_1,\mu\cup\nu^{c}_{b+1})\quad\quad\quad}&
% 	}
% \]
Then \[\dim \cH(K_1,a,\mu)=\dim \cH(K_1,a+1)+1,\]\[\aand \dim \cH(K_2,b)=\dim \cH(K_2,b+1)+1\]by Proposition \ref{prop: shape of the dimension sequence}. Then Corollary \ref{cor: more triangles 2} and Remark \ref{rem: BS notation} imply that\[F^{\infty}_{a}(K_1,01)\neq 0 \aand F^{\infty}_{b}(K_2,01)\neq 0.\]Then Lemma \ref{lem: cobordism injective} implies that
\[
F^{\infty}_{a+b}(\widebar{K}\# K,00) \neq 0.
\]
Thus, we conclude from \eqref{eq: W-shaped slice} that $a+b\leq 0$, which is equivalent to
\[
\nkp{K} - \nkm{K}\leq 1.
\]
Since the two invariants have the same parity, we conclude that $\nkp{K} = \nkm{K}$.
\epf

\subsection{\texorpdfstring{$\mathbb{Z}/4$}{Z/4}-grading and the first two cases}\label{subsec: Z/4}
Finally, we study the difference between $\dim \cH(\nk{K})$ and $\dim \cH(\nk{K},\mu)$. From \eqref{eq: homology of surgery}, the two dimensions coincide when $\nk{K}$ is odd. From the first two cases in Proposition \ref{prop: shape of the dimension sequence}, the difference is at most two when $\nk{K}$ is even. In the following, we use the $\Z/4$-grading on instanton homology to show that the difference must be two.

From \cite[\S 7.3]{scaduto2015instantons}, \cite[\S 4]{SS2018quasi} (see also \cite[\S 2.2]{Froyshov2002equi}), for any closed oriented $3$-manifold $Y$, there exists an \emph{absolute} $\Z/4$-grading on $I^\sharp(Y)$ and the grading shift of the cobordism map can be calculated by the traditional topological invariant of the cobordism via \cite[Equation (7.1)]{scaduto2015instantons}. For a bundle set $\omega\subset Y$ such that $[\omega]\neq 0\in H_1(Y;\F_2)$, there is only a \emph{relative} $\Z/4$-grading, and a choice of spin structure on $Y$ determines an absolute lift.

Recall that the set of spin structures is an affine space over $H^2(Y;\F_2)\cong H_1(Y;\F_2)$. From \eqref{eq: homology of surgery} and \cite[p. 189]{Kirbycal1999}, the two spin structures on $S^3_{2k}(K)$ for $k\in\Z$ are characterized by the property that one $\fs_0$ extends over the surgery cobordisms\begin{equation}\label{eq:spin s0}
    W^{2k}_{\infty}:S^3_{2k}(K)\to S^3\aand W^{\infty}_{2k}:S^3\to S^3_{2k}(K),
\end{equation}and the other $\fs_1$ extends over the surgery cobordisms\begin{equation}\label{eq: spin s1}
    W^{2k}_{2k+1}:S^3_{2k}(K)\to S^3_{2k+1}(K)\aand W^{2k-1}_{2k}: S^3_{2k-1}(K)\to S^3_{2k}(K).
\end{equation}To compare with \cite[Corollary 5.2]{alfieri2020framed}, we always choose $\fs_1$ to lift the absolute $\Z/4$-grading on \[\cH(2k,\mu)=I^\sharp(S^3_{2k}(K),\mu;\bK).\]This choice is not essential, as the absolute $\Z/4$-grading from $\fs_0$ differs from that of $\fs_1$ by $2$. Then the grading shifts of surgery cobordisms are calculated in the following lemma.

\blem\label{lem: Z/4 grading shift}
Suppose $k\in \Z$ and suppose the absolute $\Z/4$-grading on $\cH(2k,\mu)$ is determined by $\fs_1$ from 
\eqref{eq: spin s1}. Then the grading shifts of the surgery cobordism maps in Remark \ref{rem: BS notation} are listed as follows, where $(i)$ denotes the grading shift $i$.

\begin{equation*}
	\xymatrix{
	\cH(2k-1)\ar[rr]^{(i_1)}&& \cH(2k)\ar[dl]^{(i_2)}\\
	&\cH(\infty)\ar[lu]^{(i_3)}&
	}
    \xymatrix{
	\cH(2k)\ar[rr]^{(i_4)}&& \cH(2k+1)\ar[dl]^{(i_5)}\\
	&\cH(\infty)\ar[lu]^{(i_6)}&
	}
\end{equation*}
\begin{equation*}
	\xymatrix{
	\cH(2k-1)\ar[rr]^{(j_1)}&& \cH(2k,\mu)\ar[dl]^{(j_2)}\\
	&\cH(\infty)\ar[lu]^{(j_3)}&
	}
    \xymatrix{
	\cH(2k,\mu)\ar[rr]^{(j_4)}&& \cH(2k+1)\ar[dl]^{(j_5)}\\
	&\cH(\infty)\ar[lu]^{(j_6)}&
	}
\end{equation*}
where\[(i_1,i_2,i_3,i_4,i_5,i_6)=\begin{cases}
    (0,0,3,0,2,1)&\text{when }k>0;\\
    (3,2,2,2,2,3)&\text{when }k=0;\\
    (0,1,2,0,3,0)&\text{when }k<0;
\end{cases}\]\[(j_1,j_2,j_3,j_4,j_5,j_6)=\begin{cases}
    (0,2,1,0,0,3)&\text{when }k>0;\\
    (3,0,0,2,0,1)&\text{when }k=0;\\
    (0,3,0,0,1,2)&\text{when }k<0;
\end{cases}\]
\elem
\bpf
The results for $i_1,\dots,i_6$ are from \cite[Corollary 5.2]{alfieri2020framed}, which uses the facts that the grading shift for a cobordism with spin structure does not depend on the bundle set and the sum of grading shifts of three maps in the exact triangles is $-1\pmod 4$. To compute $j_1,\dots,j_6$, we can again use those two facts to obtain\begin{equation}\label{eq: grading formula}j_1=i_1,~j_2=i_2+2~,~j_1+j_2+j_3=-1,~j_4=i_4,~j_5=i_5+2,j_4+j_5+j_6=-1\pmod 4,
\end{equation}where we use the characterization of the spin structures in \eqref{eq:spin s0} and \eqref{eq: spin s1}. One can double check the results by computing $j_3$ and $j_5$ directly from \cite[Equation (7.1)]{scaduto2015instantons}, as there are no bundle sets on those cobordisms.
\epf
\bprop\label{prop: V and W}
If $\nk{K}$ is even, then we have
	\begin{equation*}
	    \big |\dim \cH(\nk{K})-\dim \cH(\nk{K},\mu)\big|=2.
	\end{equation*}
\eprop
\bpf
Suppose $2k=\nk{K}$. From Propositions \ref{prop: shape of the dimension sequence} and \ref{prop: differed at most by 2}, we know that \begin{equation}\label{eq: dim sign}
    \dim \cH(2k-1)=\dim \cH(2k-1)=\dim \cH(2k)\pm 1=\dim \cH(2k,\mu)\pm 1.
\end{equation}
From Lemma \ref{lem: Z/4 grading shift} and the fact that $\cH(\infty)\cong \bK$ is supported in grading $0$, if the sign in \eqref{eq: dim sign} is determined, then we can calculate the dimensions of the $(0,1,2,3)$-graded summand of $\cH(2k+1)$ from that of $\cH(2k-1)$ by passing through either $\cH(2k)$ or $\cH(2k,\mu)$. It turns out that the signs for $\cH(2k)$ and $\cH(2k,\mu)$ must be different so that the two approaches provide the same answer. This indeed follows from \eqref{eq: grading formula}. 

For example, we compute explicitly to exclude the case $k>0$ and \[\dim \cH(2k)=\dim \cH(2k,\mu)=\dim \cH(2k\pm 1)-1.\]Suppose the grading $(0,1,2,3)$ summands of $\cH(2k-1)$ and $\cH(2k+1)$ are $(a,b,c,d)$ and $(a',b',c',d')$, respectively. Then Lemma \ref{lem: Z/4 grading shift} implies that $(a',b',c',d')=(a,b,c+1,d-1)$ when passing through $\cH(2k)$ and $(a',b',c',d')=(a+1,b-1,c,d)$ when passing through $\cH(2k,\mu)$, which leads to a contradiction.
\epf

\section{Rational surgeries}\label{sec: rational surgeries}
In this section, we deal with the rational surgeries. Following \eqref{eq: H notation}, we fix a knot $K\subset S^3$ and a field $\bK$, and write\[\begin{aligned}
\cH(r,\omega)&=\cH(K,r,\omega)=I^\sharp(S^3_r(K),\omega;\bK),\\
    \cH(r)&=\cH(K,r)=I^\sharp(S^3_r(K),\omega=0;\bK)\\\aand \tcH(r)&=\tcH(K,r)=I^\sharp(S^3_r(K),\omega=\wti{K}_r;\bK)\cong I^\sharp(S^3_r(K),\mu;\bK).
\end{aligned}\]Similar to the case of $\bK=\C$ as in \cite[Theorem 4.6]{baldwin2020concordance} and \cite[Theorem 7.1]{baldwin2022concordanceII}, an important ingredient of the proof is the following lemma from \cite[\S 4]{baldwin2020concordance}.
\begin{lemma}\label{lem: existence of smaller surgery slopes}
	Suppose $p_0$ and $q_0$ are co-prime integers satisfying $p_0\neq 0$ and $|q_0|>1$. Suppose $r_0=p_0/q_0\in (k,k+1)$ for some integer $k$. Then there exist $r_i =p_i/q_i$ for $i=1,2,3$ that satisfy the following conditions.
	\begin{itemize}
    \item For $i=1,2,3$, $p_i$ and $q_i$ are co-prime, possibly zero integers, such that $p_i$ and $q_i$ have the same signs with $p_0$ and $q_0$, respectively, when they are not zero.
	\item $r_1,r_2 \in [k, k+1]$.
	\item $p_1+p_2 = p_0$ and $q_1+q_2 = q_0$.
    \item $p_3 = {\rm sign}(p_0)\cdot |p_1-p_2|$ and $q_3 = {\rm sign}(q_0) \cdot |q_1-q_2|$
    \item $(r_0,r_1,r_2,r_3)$ fits into the two slope triads as in \eqref{eq: extended triad}, denoted by
    \begin{equation}\label{eq: traids of slopes}
	\xymatrix{
		&r_0\ar[dr]&\\
		r_2\ar@<-2pt>[rr]\ar[ur] && r_1\ar[dl]\ar@<-2pt>[ll]\\
		&r_3\ar[lu]&
	}
\end{equation}
	\end{itemize}
\end{lemma}
\brem
Lemma \ref{lem: existence of smaller surgery slopes} can be illustrated by the Farey tessellation of the hyperbolic plane, where the rationals at the vertices of each triangle (with hyperbolic geodesics as edges) corresponds to one choice of $r_0$, $r_1$, and $r_2$. Examples of $(r_0,r_1,r_2,r_3)$ are \[(1/2,0/1,1/1,1/0),~(2/3,1/1,1/2,0/1),~(1/3,1/2,0/1,1/1),~(-1/2,0/1,-1/1,1/0)\]drawn in the following diagram.    \begin{equation}\label{eq: examples of slope}
	\xymatrix{
		&\frac{1}{3}\ar[rr]&&\frac{1}{2}\ar@<-2pt>[dr]\ar[rr]\ar@<-2pt>[dl]&&\frac{2}{3}\ar[dl]\\
		\frac{-1}{2}\ar[rr]&&\frac{0}{1}=0\ar@<-2pt>[dr]\ar@<-2pt>[dl]\ar@<-2pt>[rr]\ar@<-2pt>[ur]\ar[ul] && \frac{1}{1}=1\ar[dl]\ar@<-2pt>[ll]\ar@<-2pt>[ul]&\\
		&\frac{-1}{1}=-1\ar[ul]\ar@<-2pt>[ur]&&\frac{1}{0}=\infty\ar@<-2pt>[ul]\ar[ll]&&
	}
\end{equation}Note that the parallelogram with vertices $0,1,\infty,-1$ is not an example of Lemma \ref{lem: existence of smaller surgery slopes}, but is still an example of \eqref{eq: extended triad}.
\erem
The main result of this section is the following.
\bprop\label{prop: dimension formula}
	Suppose $p$ and $q$ are co-prime integers with $q>1$. Suppose that $\nk{K}$ and $r_{\bK}(K)$ come from Proposition \ref{prop: dim formula for integers}. Then we have
	\begin{equation}\label{eq: dim formula rational}
	    \dim \cH(p/q,0)=\dim \cH(p/q,\mu)= q\cdot r_{\bK}(K) + |p-q\cdot \nk{K}|.
	\end{equation}
\eprop
\begin{proof}
	The proof comes from the induction with the help of Lemma \ref{lem: existence of smaller surgery slopes}. From Proposition \ref{prop: dim formula for integers}, the formula \eqref{eq: dim formula rational} also holds for $p=0$ or $|q|=1$ except for the case when $K$ is W-shaped over $\bK$ and $p/q=\nu^\sharp_{\bK}(K)$, which is the initial step in the induction. For short, we write \[M=\nu^\sharp_{\bK}(K).\]

    We first deal with the case \begin{equation}\label{eq: case 1}
        p/q\not\in \Z\aand p/q\not\in (M-1,M+1).
    \end{equation}
    
    We apply Lemma \ref{lem: existence of smaller surgery slopes} to $p_0/q_0$ and obtain $p_i/q_i$ for $i=1,2,3$. The induction hypothesis is that $p_i/q_i$ for $i=1,2,3$ all satisfy \eqref{eq: dim formula rational}. Then the exact triangles in Proposition \ref{prop: surgery exact triangle} and the dimension counting imply that the four triangles for $(r_2,r_1,r_3)$ split and hence either \[F^{r_1}_{r_3}(\ep)=0\text{ for all }\ep\text{ or }F^{r_3}_{r_2}(\ep)=0\text{ for all }\ep,\]where $\ep\in\{00,01,10,11\}$. From Lemma \ref{lem: blow-up many cases}, we have $F^{r_1}_{r_2}(\ep)=0$ for all $\ep$. Thus, the four triangles in Proposition \ref{prop: surgery exact triangle} for $(r_0,r_1,r_3)$ also split and the dimension counting implies that $p_0/q_0$ also satisfies \eqref{eq: dim formula rational}.

    Then we deal with the case \begin{equation}\label{eq: case 2}
        p/q\in (M-1,M)\cup (M,M+1).
    \end{equation}

    From Proposition \ref{prop: dim formula for integers}, if $M$ is odd, then $M$ also satisfies \eqref{eq: dim formula rational} and the proof of the case in \eqref{eq: case 1} applies. If $M$ is even, then exactly one of $\cH(M)$ and $\tcH(M)$ satisfies \eqref{eq: dim formula rational}. The proofs for the two cases are similar by using different exact triangles. Alternatively, one can add $\mu\times I$ into the bundle sets of all cobordism maps to switch one case to the other. We only consider the case that $\tcH(M)$ satisfies \eqref{eq: dim formula rational}, partly because of the expectation that any knot is W-shaped over any field $\bK$ in Question \ref{ques: W-shaped}.

    Furthermore, from the proof of the case in \eqref{eq: case 1} and Lemma \ref{lem: existence of smaller surgery slopes}, it suffices to show that \eqref{eq: dim formula rational} holds for all \begin{equation}\label{eq: exceptionals}
        p/q\in \left\{M\pm \frac{1}{n+1},M\pm \frac{2}{2n+1}\right\}_{n\in\Z_+}
    \end{equation}because then we obtain the remaining results by induction. When $M=0$, we can consider the diagram in \eqref{eq: examples of slope}. In the following, we only prove the case when $M=0$ and the signs in \eqref{eq: exceptionals} are positive. The general case for $M$ even and positive signs in \eqref{eq: exceptionals} is obtained by adding $M$ to all slopes except $\infty$. The general case for $M$ even and negative signs in \eqref{eq: exceptionals} is obtained by applying the result to the mirror knot $\widebar{K}$ and by using the fact that\[\dim \cH(K,r)=\dim \cH(\widebar{K},r)\aand \dim \tcH(K,r)=\dim \tcH(\widebar{K},r).\]Indeed, one may also apply the strategy for positive signs to negative signs, though the details of the proof would be different. From now on, we assume that $n\in\Z_+$.

    From the above reduction, we assume that $\tcH(0)$ satisfies \eqref{eq: dim formula rational}, and then by Proposition \ref{prop: dim formula for integers}, we have \begin{equation}\label{eq: assumption case 2}
        \dim \cH(0)=\dim \tcH(0)+2=\dim \cH(1)+1=\dim \tcH(1)+1.
    \end{equation}
    To be clear, we split the rest of the proofs into different cases.

    \indent{\bf Case 1}. $\cH(1/n)$ and $\tcH(1/n)$.
    
    By Proposition \ref{prop: surgery exact triangle}, the following maps vanish
    \[F^\infty_0(01)\aand F^\infty_0(11).\]
    % \[F^\infty_0(01),~F^\infty_0(11),\aand F^0_{\infty}(10).\]
    
    % \[F^1_\infty(10),~F^1_\infty(01),~F^\infty_{-1}(10),~F^\infty_{-1}(01),~F^\infty_0(01),~F^\infty_0(11),~F^0_{\infty}(10),\aand F^0_{\infty}(11).\]
    From Lemma \ref{lem: blow-up many cases}, we have    \[F^{1/n}_0(01)=\pm F^{1/(n-1)}_0(11)\circ F^{1/n}_{1/(n-1)}(01)\aand  F^{1/n}_0(11)=\pm F^{1/(n-1)}_0(01)\circ F^{1/n}_{1/(n-1)}(10)\]
    % \[\aand F^0_{-1/n}(10)=\pm F^{-1/(n-1)}_{-1/n}(10)\circ F^0_{-1/(n-1)}(11),\]
    Hence by induction, the following maps vanish for all $n\in\Z_+$ \[F^{1/n}_0(01)\aand  F^{1/n}_0(11).\]
    % \[F^{1/n}_0(01),~ F^{1/n}_0(11),\aand F^0_{-1/n}(10).\]
    By Proposition \ref{prop: surgery exact triangle} and dimension counting, we know that $\cH(1/n)$ and $\tcH(1/n)$ satisfy \eqref{eq: dim formula rational}.

    \indent{\bf Case 2}. $\cH(2/(2n+1))$.
    
    From Lemma \ref{lem: blow-up many cases}, we have\[F^{1/n}_{1/(n+1)}(10)=\pm F^0_{1/(n+1)}(10)\circ F^{1/n}_0(11)=0.\]
    
% \[F^{1/n}_{1/(n+1)}(10)=\pm F^0_{1/(n+1)}(10)\circ F^{1/n}_0(11)=0,~F^{1/n}_{1/(n+1)}(01)=\pm F^0_{1/(n+1)}(11)\circ F^{1/n}_0(01)=0.\]    
    % \[\aand F^{-1/(n+1)}_{-1/n}(10)=\pm F^0_{-1/n}(10)\circ F^{-1/(n+1)}_0(11)=0.\]
    
    By Proposition \ref{prop: surgery exact triangle}, we know that $\cH( 2/(2n+1))$ satisfies \eqref{eq: dim formula rational}.

    \indent{\bf Case 3}. $\tcH(2/(2n+1))$.
    
    Note that from \eqref{eq: homology of surgery}, we do not necessarily have\[\dim \tcH(\pm 2/(2n+1))=\dim \cH(\pm 2/(2n+1))\]The result that $\tcH(\pm 2/(2n+1))$ satisfies \eqref{eq: dim formula rational} needs some new ingredient from Lemma \ref{lem: embedded sphere} \eqref{t4}. We will use exact triangles from Proposition \ref{prop: surgery exact triangle} freely.
    
    \indent{\bf Case 3.1}. $\tcH(2/3)$.
    
    We start with the case of $2/3$ as shown in \ref{eq: examples of slope} and then prove the general case. By the triangle, it suffices to show that \[F^{1}_{1/2}(00)=0.\]By Lemma \ref{lem: blow-up many cases}, we have \[F^{1}_{1/2}(00)=\pm F^{0}_{1/2}(00)\circ F^1_0(00)=\pm F^0_{1/2}(00)\circ F^\infty_0(00)\circ F^1_\infty(00)\]\[\aand F^1_0(10)=\pm F^\infty_0(10)\circ F^1_\infty(11).\]From \eqref{eq: assumption case 2}, we know that\[\rk F^1_\infty(00)= \rk F^\infty_0(00)=\rk F^\infty_0(10)= \rk F^1_\infty(11)=1\]\[\img F^1_0(00)=\img F^\infty_0(00), \aand \img F^1_0(10)= \img F^\infty_0(10).\]From Proposition \ref{prop: dim formula for integers}, we know that \[\dim \tcH(2)=\dim \cH(1)+1\aand \rk F^2_\infty(10)=1.\]
    Since $\dim \cH(\infty)=1$, we know that
    \[
        \rk  \left(F^\infty_{0}(00)\circ F^2_{\infty}(10)\right)=1.
    \]
    From Lemma \ref{lem: embedded sphere} \eqref{t4}, we have
    \[
        F^\infty_{0}(10)\circ F^2_{\infty}(11)=\pm F^\infty_{0}(00)\circ F^2_{\infty}(10).
    \]
    Hence we have
    \[
        \img F^\infty_0(10)=\img F^\infty_0(00).
    \]
    In summary, we conclude
    \begin{equation}\label{eq: images}
        \img F^1_0(00)=\img F^\infty_0(00)= \img F^\infty_0(10)=\img F^1_0(10).
    \end{equation}
    Thus, we have \[\begin{aligned}
        \{0\}=&\img\left(F^0_\frac{1}{2}(00)\circ F^1_0(10)\right)\\
        =& \img \left(F^0_{1/2}(00)\circ F^1_0(00)\right)\\
        =& \img F^{1}_{1/2}(00),
    \end{aligned}\]which concludes the case of $2/3$.

    \indent{\bf Case 3.2}. $\tcH(2/(2n+1))$.
    
    For $2/(2n+1)$, we need to show that \[F^{1/n}_{1/(n+1)}(00)=0.\]By applying Lemma \ref{lem: blow-up many cases} for many times, we have \[\begin{aligned}
        F^{1/n}_{1/(n+1)}(00)=&\pm F^0_{1/(n+1)}(00)\circ F^{1/n}_0(00)\\=&\pm F^0_{1/(n+1)}(00)\circ \left(F^\infty_0(00)\circ F^1_\infty(00)\circ \cdots\circ F^{1/n}_{1/(n-1)}(00)\right)
    \end{aligned}\]\[\aand F^{1/n}_{0}(10)=\pm F^\infty_0(10)\circ \left(F^1_\infty(11)\circ \cdots F^{1/n}_{1/(n-1)}(11) \right).\]The facts that $1/n$ and $1/(n+1)$ satisfy \eqref{eq: dim formula rational} from Case 1 and the dimension equalities in \eqref{eq: assumption case 2} imply that\[\rk F^{1/n}_0(00)=\rk F^{1/n}_{0}(10)=1.\]Then we have\[\begin{aligned}
        \{0\}=&\img\left(F^0_{1/(n+1)}(00)\circ F^{1/n}_0(10)\right)\\
        =& \img \left(F^0_{1/(n+1)}(00)\circ F^\infty_0(10)\right)\\
        =& \img \left(F^0_{1/(n+1)}(00)\circ F^\infty_0(00)\right)\\
        =& \img \left(F^0_{1/(n+1)}(00)\circ F^{1/n}_0(00)\right)\\
        =& \img F^{1/n}_{1/(n+1)}(00),
    \end{aligned}\]where the third equality is from \eqref{eq: images}, which concludes the case of $2/(2n+1)$.
    % \[F^0_{-1}(10)=\pm F^{\infty}_{-1}(10)\circ F^0_{\infty}(11)=0,\aand F^0_{-1}(11)=\pm F^{\infty}_{-1}(01)\circ F^0_{\infty}(10)=0.\]
    % \[F^{1/n}_0(01)=\pm F^{1/(n-1)}_0(11)\circ F^{1/n}_{1/(n-1)}(01),~ F^{1/n}_0(11)=\pm F^{1/(n-1)}_0(01)\circ F^{1/n}_{1/(n-1)}(10),\]
    % \[F^{1/n}_0(10)=\pm F^{1/(n-1)}_0(10)\circ F^{1/n}_{1/(n-1)}(11),\]
\end{proof}
\brem
For the proof of the case in \eqref{eq: case 1}, one may also use Lemma \ref{lem: embedded sphere}\eqref{t2} instead of Lemma \ref{lem: embedded sphere}\eqref{t3} to carry out the proof. However, the authors have not found a proof of the case in \eqref{eq: case 2} without using Lemma \ref{lem: embedded sphere}\eqref{t4}.
\erem

\section{\texorpdfstring{$SU(2)$}{SU(2)}-representations}\label{sec: SU(2)-representations}

In this section, we focus on $\bK=\F_2$ and study $SU(2)$-representations on knot surgeries. The following lemma connects instanton homology to $SU(2)$-representations.
\begin{lemma}\label{lem: nondegenerate}
    Suppose $K\subset S^3$ is a knot and suppose $p/q\in\mathbb{Q}\backslash\{0\}$ for co-prime integers $p$ and $q$. If $S^3_{p/q}(K)$ is $SU(2)$-abelian and $p\in\{a^e,2a^e\}$ for some prime number $a$ and natural number $e$, then we have \[I^\sharp(S^3_{p/q}(K);\Z)\cong \Z^{|p|}.\]In particular, the knot $K$ is an instanton L-space knot over any field $\bK$.
\end{lemma}
\bpf
The case $p=a^e$ follows directly from \cite[Corollary 4.8]{baldwin2018stein}, and the case $p=2a^e$ is a slight generalization. By \cite[Theorem 4.6 and Proposition 4.7]{baldwin2018stein}, this case reduces to the fact that for the Alexander polynomial $\Delta_K(t)$ of $K$, any $2a^e$-th root of unity $\omega$ satisfies $\Delta_K(\omega ^2)\neq 0$. 

Suppose $\omega$ is a $2a^e$-th root of unity such that $\Delta_K(\omega^2) = 0$. Let $\eta = \omega^2$. Then $\eta$ is a $a^e$-th root of unity and $\Delta_K(\eta) = 0$. Similar to the proof of \cite[Corollary 9.2]{baldwin2019lspace}, let $\Phi(t)$ be the cyclotomic polynomial associated to $\eta$. We have $\Phi(t) \mid \Delta_K(t)$. Hence $p= \Phi(1)\mid \Delta_K(1)=1$, which leads to a contradiction.
\epf

% When $\bK=\F_2$ in Proposition \ref{prop: dim formula for integers}, Ghosh--Miller-Eismeier \cite{SGMME} proved the following stronger result.
% \bthm[{\cite[Theorem 1.3]{SGMME}}]\label{thm: F_2 result}
% Any knot $K\subset S^3$ is W-shaped over $\F_2$. Moreover, the integers $\nu^\sharp_{\F_2}(K)$ and $r_{\F_2}(K)$ are both divisible by $4$.
% \ethm
Now we prove the main theorems.
\bthm\label{thm: larger than nu}
Suppose $K$ is a nontrivial knot of genus $g~(\ge 1)$ and suppose $p/q\in (0,\infty)$ with $q\ge 1$, $\gcd(p,q)=1$, and $p\in\{a^e,2a^e\}$ for some prime number $a$ and natural number $e$. If $S^3_{p/q}(K)$ is $SU(2)$-abelian, then the knot $K$ is an instanton L-space knot over any field $\bK$,\[r_{\bK}(K)=\nk{K}\ge \nu^\sharp_\C(K)=2g-1\aand p/q\ge \nk{K}.\]Moreover, for $\bK=\F_2$, we have\[\nu^\sharp_{\F_2}(K)\ge \nu^\sharp_\C(K)+1=2g.\]
\ethm
\bpf
Suppose $S^3_{p/q}(K)$ is $SU(2)$-abelian with $p\in\{a^e,2a^e\}$ and $q\ge 1$. Then Lemma \ref{lem: nondegenerate} implies that \begin{equation}\label{eq: dim p}
    \dim I^\sharp(S^3_{p/q}(K);\bK)=p
\end{equation}for any field $\bK$. Hence $K$ is an instanton L-space knot over any field $\bK$. From \cite[Theorem 1.18]{baldwin2020concordance}, we have \[r_{\C}(K)=\nu^\sharp_{\C}(K)=2g-1.\]

From Theorem \ref{thm: dimension formula, main}, if $p/q< \nu^\sharp_{\bK}(K)$, then \[p=q\cdot r_{\bK}(K) - (p-q\cdot \nk{K}),\]which implies \[\frac{p}{q}=\frac{r_{\bK}(K)+\nu^\sharp_{\bK}(K)}{2}\ge \frac{|\nu^\sharp_{\bK}(K)|+\nu^\sharp_{\bK}(K)}{2}\ge \nu^\sharp_{\bK}(K),\]where the first inequality follows from Remark \ref{rem: inequality}. This leads to a contradiction. 

Thus, we obtain $p/q\ge  \nu^\sharp_{\bK}(K)$. From Theorem \ref{thm: dimension formula, main} and \eqref{eq: dim p}, we obtain\[p=q\cdot r_{\bK}(K) + p-q\cdot \nk{K},\]or equivalently \begin{equation}\label{eq: same r nu}
    r_{\bK}(K)=\nk{K}.
\end{equation}

From the universal coefficient theorem, we have \[\dim I^\sharp(S^3_1(K);\bK)\geq\dim I^\sharp(S^3_1(K);\C)\]for any field $\bK$.
Applying Theorem \ref{thm: dimension formula, main} and \eqref{eq: same r nu} to $p/q=1$ and $\bK=\C$, we obtain that 
\begin{equation}\label{eq: inequality nu}
    \begin{aligned}
    \nu^\sharp_{\bK}(K)+|1-\nu^\sharp_{\bK}(K)|=&|r_{\bK}(K)|+|1-\nu^\sharp_{\bK}(K)|\\\ge&r_{\C}(K)+|1-\nu^\sharp_{\C}(K)|\\=&\nu^\sharp_{\C}(K)+|1-\nu^\sharp_{\C}(K)|\\=&2\nu^\sharp_{\C}(K)-1\\\ge&2(2g(K)-1)-1\\=&1.
\end{aligned}
\end{equation}
If $\nu^\sharp_{\bK}(K)\le 1$, then all inequalities in \eqref{eq: inequality nu} should be equalities and $\nu^\sharp_{\C}(K)=g(K)=1$. If $\nu^\sharp_{\bK}(K)>1$, then the left-hand-side of \eqref{eq: inequality nu} equals to $2\nu^\sharp_{\bK}(K)-1$ and \eqref{eq: inequality nu} implies that \[\nu^\sharp_{\bK}(K)\ge \nu^\sharp_{\C}(K)= 2g(K)-1.\]

Furthermore, for $\bK=\F_2$, \cite[Theorem 1.1]{LY20255surgery} implies that\[\dim I^\sharp(S^3_1(K);\F_2)>\dim I^\sharp(S^3_1(K);\C).\]Hence a modification of \eqref{eq: inequality nu} implies that \[\nu^\sharp_{\F_2}(K)>\nu^\sharp_{\C}(K).\]Since both invariants are integers, we have \[\nu^\sharp_{\F_2}(K)\ge \nu^\sharp_{\C}(K)+1=2g.\]

% Finally, Theorem \ref{thm: F_2 result} implies that for $p/q= \nu^\sharp_{\F_2}(K)$, we have\[\dim I^\sharp(S^3_{p/q}(K);\F_2)=r_{\bF_2}(K)+2=\nu^\sharp_{\F_2}(K)+2,\]which contradicts \eqref{eq: dim p}. Hence $p/q> \nu^\sharp_{\F_2}(K)$.
\epf
\bthm\label{thm: abelian}
Suppose $K$ is a nontrivial knot and suppose $p/q\in (2,6)$ with $q\ge 1$, $\gcd(p,q)=1$, and $p\in\{a^e,2a^e\}$ for some prime number $a$ and non-negative integer $e$. Then $S^3_{p/q}(K)$ is $SU(2)$-abelian only when $K=T_{2,3}$ and \[\frac{p}{q}\in \big\{6-\frac{1}{n}\big\}_{n\in\Z_+}.\]
\ethm
\bpf
Note that $K$ is an instanton L-space knot (over $\C$) by Theorem \ref{thm: larger than nu}.

When $g\ge 3$, then Theorem \ref{thm: larger than nu} implies that\[p/q\geq\nu^\sharp_{\F_2}(K)\ge 2g\ge 6.\]
% In the assumption of Theorem \ref{thm: larger than nu}, we prove that if $K\neq T_{2,3}$, then $p/q\ge 6$.

When $g\le 2$, the results in \cite{BS2022khovanov,farber2024fixed} imply that $K$ must be either $T_{2,3}$ or the torus knot $T_{2,5}$. From \cite[Proposition 4.3]{SZ2022surgery}, only $T_{2,3}$ and $p/q\in\{6,6-1/n\}_{n\in\Z_+}$ can be the candidate when $p/q\in (2,6)$.
% Since $\nu^\sharp_{\F_2}(K)$ is divisible by $4$ from Theorem \ref{thm: F_2 result}, we have \[\nu^\sharp_{\F_2}(K)\ge 8.\]Then Theorem \ref{thm: larger than nu} implies that\[p/q>\nu^\sharp_{\F_2}(K)\ge 8.\]
\epf

\section{Bypass exact triangle and genus one knots}\label{sec: genus 1}
In this section, we fix the proof of instanton bypass exact triangle in \cite[\S 4]{BS2022khovanov} and then prove Proposition \ref{prop: genus 1}.

Because of Remark \ref{rem: dependence wrong} and Example \ref{exmp: dependence}, one can only use Lemma \ref{lem: cob map depending on H_2} to identify cobordism maps with different bundle sets. In the proof of the instanton bypass exact triangle, especially \cite[Formula (19)]{BS2022khovanov}, Baldwin--Sivek claimed the following\begin{equation}\label{eq: BS bundle}
    I(W_1,\nu)=I(W_1,\widebar{\kappa}_1)\aand I(W_2,\nu)=I(W_2,\kappa_2),
\end{equation}where $W_1:Y_1\to Y_2$ and $W_2:Y_2\to Y_3$ are two consecutive surgery cobordisms in a surgery exact triangle for some knot $K\subset Y_1$, with the extra bundle set in the incoming end of $W_1$, $\nu=\omega\times I$ for some $\omega=\alpha\cup\eta$ disjoint from $K$, the bundle set $\widebar{\kappa}_1$ is the union of some punctured torus $\widehat{T}\subset Y_1$ with $K$ framed by $\partial \widehat{T}$ and the bundle set from the triangle, the bundle set $\widebar{\kappa}_2$ is the one from the triangle. This is indeed the exact triangle \eqref{eq: surgery exact triangle, more 2 3}. Thus, we know $\widebar{\kappa}_1$ is the union of $\nu$ and a torus and $\kappa_2$ is just $\nu$, which implies that the second equation in \eqref{eq: BS bundle} is trivial.

Baldwin--Sivek used the fact that $b_1(Y_2)-1=b_1(Y_1)=b_1(Y_3)$ to obtain \eqref{eq: BS bundle}, which is insufficient because similar phenomenon might happen as in Example \ref{exmp: dependence}. However, with a little more work, we show that \eqref{eq: BS bundle} still holds because of the existence of the punctured torus. Note that Example \ref{exmp: dependence} is based on a Seifert surface of genus larger than $1$.

Since the manifold $Y_1$ and the bundle set $\omega$ are obtained from a closure of balanced sutured manifold and the punctured torus $\widehat{T}$ is constructed from the fact that $K$ intersects the suture twice \cite[Figure 7 and p. 921]{BS2022khovanov}, the bundle set $\eta$ intersects $\widehat{T}$ at one point and hence $|\omega\cap \widehat{T}|=1$. 

Moreover, note that we only need \eqref{eq: BS bundle} over the coefficient field $\bK=\C$ and on the $(+2)$-generalized eigenspace of $\mu(\pt)$. Since $\mu(\pt)$ is defined in \cite[\S 7.3]{donaldson2002floer} whenever $\mathrm{char}(\bK)\neq 2$ (see also \cite{DK1992instanton,KM1995structure} for the closed case), we write \[I(Y,\omega;\bK)_{\pm 2}=\mathrm{colim}_{N}\ker \left((\mu(\pt)\pm 2)^N\right)\subset I(Y,\omega;\bK)\]for the $(\pm 2)$-generalized eigenspace of $\mu(\pt)$ over such a field $\bK$. From \cite[Theorem 9]{Froyshov2002equi}, when $(Y,\omega)$ is nontrivial admissible, we know \begin{equation*}\label{eq: vanish}
    (\mu(\pt)-4)^{N}=(u-64)^{N}=0
\end{equation*}for some large integer $N$. From Jordan-Chevalley decomposition, since $+2\neq -2\in \bK$ with $\mathrm{char}(\bK)\neq 2$, we have the generalized eigenspace decomposition\[I(Y,\omega;\bK)=I(Y,\omega;\bK)_{+2}\op I(Y,\omega;\bK)_{-2}.\]Since $I(Y,\omega;\bK)$ inherits a relative $\Z/8$-grading and $\mu=\mu(\pt)$ is degree $4$, for homogeneous elements $x,y\in I(Y,\omega;\bK)$ with $\gr(x)=\gr(y)+4$, we have\[(\mu+ 2)^N(x+y)=0\text{ if and only if }(\mu-2)^N(x-y)=0,\]because both equations are equivalent to\[(\mu^N+c_{N-2}\cdot \mu^{N-2}+\cdots)(x)+(c_{N-1}\cdot \mu^{N-1}+c_{N-3}\cdot \mu^{N-3}+\cdots)(y)\]\[\aand (\mu^N+c_{N-2}\cdot \mu^{N-2}+\cdots)(y)+(c_{N-1}\cdot \mu^{N-1}+c_{N-3}\cdot \mu^{N-3}+\cdots)(x)\]for some integers $c_{i}$. Thus, we have\[\dim I(Y,\omega,\bK)_{\pm 2}=\frac{1}{2}\dim I(Y,\omega,\bK).\]Since $\mu(\pt)$ commutes with the cobordism map, a cobordism $(W,\nu):(Y_0,\omega_0)\to (Y_1,\omega_1)$ between nontrivial admissible pairs induce\[I(W,\nu)_{\pm 2}: I(Y_0,\omega_0;\bK)_{\pm 2}\to (Y_1,\omega_1;\bK)_{\pm 2}\]\[\mathrm{with}~I(W,\nu)=I(W,\nu)_{+2}+I(W,\nu)_{-2}.\]

Then the following proposition and corollary fix the proof of the bypass exact triangle.

\bprop\label{prop: torus}
    Suppose $(W,\nu):(Y_0,\omega_0)\to (Y_1,\omega_1)$ is a cobordism between nontrivial admissible pairs. Suppose $T\subset W$ is an embedded torus such that $T\cdot T=0$ and $l=|\nu\cap T|$ is odd. Suppose $\bK$ is a field with $\mathrm{char}(\bK)\neq 2$, then we have
	\[
		I(W,\nu\cup T)_{\pm 2} = c_{\pm}\cdot I(W,\nu)_{\pm 2}:I(Y_0,\omega_0;\bK)_{\pm 2}\to I(Y_1,\omega_1;\bK)_{\pm 2},
	\]where $c_{\pm}\in \{\pm 1\}$.
\eprop
\bpf
Let $W_T=W\backslash {\rm int}N(T)$ and $\nu_T=\nu\cap W_T$. Since $T\cdot T=0$ and $l=|\nu\cap T|$ is odd. We have\[(N(T),\nu\cap N(T))=(T\times D^2,l\cdot D^2)\aand\]\[(W_T,\nu_T):(Y_0,\omega_0)\sqcup (T\times S^1,l\cdot S^1)\to (Y_1,\omega_1).\]
From \cite[Lemma 4]{Froyshov2002equi}, we know that\begin{equation}\label{eq: group}
    I(T\times S^1,l\cdot S^1;\bK)\cong \bK\langle x,y\rangle,
\end{equation}where $x$ and $y$ are homogeneous generators with grading difference $4$. Moreover, $\mu(\pt)$ sends $x$ to $2y$ and $y$ to $2x$. Then\[I(T\times S^1,l\cdot S^1;\bK)_{\pm 2}=\bK\langle x\pm y\rangle.\]
Note that $(T\times D^2,l\cdot D^2)$ and $(T\times D^2,(l\cdot D^2)\cup (T\times \{0\}))$ induce two homogeneous elements in \eqref{eq: group} with a grading difference of $4$ by the index formula. Without loss of generality, we assume that \[I(T\times D^2,l\cdot D^2)=a\cdot x \aand I(T\times D^2,(l\cdot D^2)\cup (T\times \{0\}))=b\cdot y\]for some $a,b\in \bK$. Under the pair of pants cobordism \[(T\times P,l\cdot P):(T\times D^2,l\cdot D^2)\sqcup (T\times D^2,l\cdot D^2)\to (T\times D^2,l\cdot D^2),\]inserting $a\cdot x$ into one summand induces the identity map and inserting $b\cdot y$ into two summands induces $\pm a\cdot x$ by Lemma \ref{lem: cob map depending on H_2}. Hence $a=\pm 1$ and $b=\pm 1$, where we implicitly use the fact that $x,y$ and relative invariants are defined over $\Z$ and hence $b^2\ge 0$. Under the generalized eigenspace decomposition, we have\[x=\frac{1}{2}(x+y)+\frac{1}{2}(x-y)\aand y=\frac{1}{2}(x+y)-\frac{1}{2}(x-y).\]Hence the proposition follows from the functoriality result.
\epf
\bcor\label{cor: reduce triangle}
Suppose $(Y,\omega,K)$ is a nontrivial admissible surgery pair such that $K$ is framed by a genus-one Seifert surface $S$. Suppose $\bK$ is a field with $\mathrm{char}(\bK)\neq 2$. If $|\omega\cap S|$ is odd, then the exact triangle \eqref{eq: surgery exact triangle, more 2 3} reduces to the following one.
\begin{equation*}\label{eq: surgery exact triangle, more 2 3 reduced}
	\xymatrix{
	I(Y,\omega;\bK)_{\pm 2}\ar[rr]^{F^\infty_0(10)\circ \bI'_S=\pm F^\infty_0(00)}&& I(Y_0(K),\omega;\bK)_{\pm 2}\ar[dl]^{F^0_1(00)}\\
	&I(Y_1(K),\omega;\bK)_{\pm 2}\ar[lu]^{\bI_S\circ F^1_\infty(00)}&
	}
\end{equation*}
\ecor
\bpf
The union of $S$ and the core disk in $W^\infty_0$ is an embedded torus satisfies the assumption of Proposition \ref{prop: torus}. Note that, the union of $S$ and the cocore disk in $W^{1}_\infty$ has self-intersection $1$ and does not satisfy the assumption of Proposition \ref{prop: torus}.
\epf
\brem
By Corollary \ref{cor: reduce triangle} and the discussion at the beginning of this section, we fix the proof of bypass exact triangle by picking up suitable closure of balanced sutured manifold. Note that the sign does not matter because the kernel and image of a map are both unchanged by adding a minus sign. The idea of studying the effect of the torus as an extra bundle set is due to John Baldwin. Note that Corollary \ref{cor: reduce triangle} only implies the exactness of the bypass exact triangle at one vertex, and we need to apply the naturality of the sutured instanton homology to prove the exactness at the other vertices, as in the original proof \cite[\S 4]{BS2022khovanov}.
\erem
\bprop\label{prop: genus one proof}
Suppose $K\subset S^3$ is a genus-one knot and suppose $\bK$ is a field with $\mathrm{char}(\bK)\neq 2$. Then \[|\nu^\sharp_{\bK}(K)|\le 1.\]
\eprop
\bpf
When $\mathrm{char}(\bK)\neq 2$, the discussion in \cite[\S 2.1]{LY2025torsion} implies that \[I^\sharp(Y,\omega;\bK)\cong I((Y,\omega)\#(T^3,S^1);\bK)_{+2}\]and the isomorphism intertwine with the cobordism map. From Proposition \ref{prop: surgery exact triangle}, we have the following exact triangle
\begin{equation*}
	\xymatrix{
	I^\sharp(S^3_1(K);\bK)\ar[rr]^{F^1_2(10)\circ \bI'_{S+D_1}}&& I^\sharp(S^3_2(K);\bK)\ar[dl]^{F^2_\infty(00)}\\
	&I^\sharp(S^3;\bK)\ar[lu]^{\bI_{S+D_1}\circ F^\infty_1(01)\quad}&
	}
\end{equation*}From Lemma \ref{lem: blow-up many cases}, we have \[F^\infty_1(01)=\pm F^0_{1}(11)\circ F^\infty_0(01).\]Note that the union of the genus-one Seifert surface $S$ and the core disk in $W^\infty_{0}$ is an embedded torus of self-intersection $0$. Moreover, it intersects the cocore disk in $W^\infty_0$ at one point. Then by Proposition \ref{prop: torus}, we have \[F^\infty_0(01)=\pm F^\infty_0(11)\circ \bI'_S.\]Proposition \ref{prop: surgery exact triangle}, we have\[\pm F^0_{1}(11)\circ F^\infty_0(11)=0\]and hence $F^\infty_1(01)=0$. Then \[\dim I^\sharp(S^3_2(K);\bK)=\dim I^\sharp(S^3_1(K);\bK)+ 1.\]Since the mirror knot $\widebar{K}$ also has a genus-one Seifert surface, the above discussion also implies that\[\dim I^\sharp(S^3_2(\widebar{K});\bK)=\dim I^\sharp(S^3_1(\widebar{K});\bK)+ 1.\]Since \[\dim I^\sharp(S^3_n(\widebar{K});\bK)=\dim I^\sharp(S^3_{-n}(K);\bK)\]for any integer $n$, by Proposition \ref{prop: dim formula for integers}, we have $|\nu^\sharp_{\bK}(K)|\le 1$.
\epf

\bibliographystyle{alpha}

\begin{thebibliography}{KMOS07}

\bibitem[ABDS22]{alfieri2020framed}
Antonio Alfieri, John~A. Baldwin, Irving Dai, and Steven Sivek.
\newblock Instanton {F}loer homology of almost-rational plumbings.
\newblock {\em Geom. Topol.}, 26(5):2237--2294, 2022.

\bibitem[AM25]{AM2025sign}
Mohammed Abouzaid and Ciprian Manolescu.
\newblock {Canonical orientations in Heegaard Floer theory}.
\newblock arXiv:2510.20062, v1, 2025.

\bibitem[BCK22]{boden2022GLPairing}
Hans~U. Boden, Micah Chrisman, and Homayun Karimi.
\newblock The {G}ordon-{L}itherland pairing for links in thickened surfaces.
\newblock {\em Internat. J. Math.}, 33(10-11):Paper No. 2250078, 47, 2022.

\bibitem[Bha24]{bhat2023newtriangle}
Deeparaj Bhat.
\newblock Surgery exact triangles in instanton theory.
\newblock {\em arXiv: 2311.04242, v2}, 2024.

\bibitem[BLSY24]{BLSY21}
John~A. Baldwin, Zhenkun Li, Steven Sivek, and Fan Ye.
\newblock Small {D}ehn surgery and {${\rm SU}(2)$}.
\newblock {\em Geom. Topol.}, 28(4):1891--1922, 2024.

\bibitem[BS18]{baldwin2018stein}
John~A. Baldwin and Steven Sivek.
\newblock Stein fillings and {SU}(2) representations.
\newblock {\em Geom. Topol.}, 22(7):4307--4380, 2018.

\bibitem[BS21]{baldwin2020concordance}
John~A. Baldwin and Steven Sivek.
\newblock Framed instanton homology and concordance.
\newblock {\em J. Topol.}, 14(4):1113--1175, 2021.

\bibitem[BS22a]{baldwin2022concordanceII}
John~A. Baldwin and Steven Sivek.
\newblock {Framed instanton homology and concordance II}.
\newblock arXiv: 2206.11531, v1, 2022.

\bibitem[BS22b]{BS2022khovanov}
John~A. Baldwin and Steven Sivek.
\newblock Khovanov homology detects the trefoils.
\newblock {\em Duke Math. J.}, 171(4):885--956, 2022.

\bibitem[BS23]{baldwin2019lspace}
John~A. Baldwin and Steven Sivek.
\newblock Instantons and {L}-space surgeries.
\newblock {\em J. Eur. Math. Soc. (JEMS)}, 25(10):4033--4122, 2023.

\bibitem[CDX20]{CDX2020polygon}
Lucas Culler, Aliakbar Daemi, and Yi~Xie.
\newblock Surgery, polygons and {${\rm SU}(N)$}-{F}loer homology.
\newblock {\em J. Topol.}, 13(2):576--668, 2020.

\bibitem[Con79]{conner1979differentiable}
Pierre~E. Conner.
\newblock {\em Differentiable periodic maps}, volume 738 of {\em Lecture Notes in Mathematics}.
\newblock Springer, Berlin, second edition, 1979.

\bibitem[DK90]{DK1992instanton}
S.~K. Donaldson and P.~B. Kronheimer.
\newblock {\em The geometry of four-manifolds}.
\newblock Oxford Mathematical Monographs. The Clarendon Press, Oxford University Press, New York, 1990.
\newblock Oxford Science Publications.

\bibitem[DMEL24]{DMML2024distancetwo}
Aliakbar Daemi, Mike Miller~Eismeier, and Tye Lidman.
\newblock {Filtered instanton homology and cosmetic surgery}.
\newblock arXiv:2410.21248, v2, 2024.

\bibitem[Don87]{Donaldson1987orientation}
S.~K. Donaldson.
\newblock The orientation of {Y}ang-{M}ills moduli spaces and {$4$}-manifold topology.
\newblock {\em J. Differential Geom.}, 26(3):397--428, 1987.

\bibitem[Don90]{Donaldson1990polynomial}
S.~K. Donaldson.
\newblock Polynomial invariants for smooth four-manifolds.
\newblock {\em Topology}, 29(3):257--315, 1990.

\bibitem[Don02]{donaldson2002floer}
S.~K. Donaldson.
\newblock {\em Floer homology groups in {Y}ang-{M}ills theory}, volume 147 of {\em Cambridge Tracts in Mathematics}.
\newblock Cambridge University Press, Cambridge, 2002.
\newblock With the assistance of M. Furuta and D. Kotschick.

\bibitem[Flo90]{floer1990knot}
Andreas Floer.
\newblock Instanton homology, surgery, and knots.
\newblock In {\em Geometry of low-dimensional manifolds, 1 ({D}urham, 1989)}, volume 150 of {\em London Math. Soc. Lecture Note Ser.}, pages 97--114. Cambridge Univ. Press, Cambridge, 1990.

\bibitem[Fre21]{Freeman2021triangle}
Jesse Freeman.
\newblock {\em {The Surgery Exact Triangle in Monopole Floer Homology with $\mathbb{Z}[i]$ Coefficients}}.
\newblock https://hdl.handle.net/1721.1/139100, 2021.

\bibitem[FRW24]{farber2024fixed}
Ethan Farber, Braeden Reinoso, and Luya Wang.
\newblock Fixed-point-free pseudo-{A}nosov homeomorphisms, knot {F}loer homology and the cinquefoil.
\newblock {\em Geom. Topol.}, 28(9):4337--4381, 2024.

\bibitem[Fy02]{Froyshov2002equi}
Kim~A. Fr\o~yshov.
\newblock Equivariant aspects of {Y}ang-{M}ills {F}loer theory.
\newblock {\em Topology}, 41(3):525--552, 2002.

\bibitem[GME]{SGMME}
Sudipta Ghosh and Mike Miller~Eismeier.
\newblock {Framed instanton homology and Fr\o yshov invariant}.
\newblock In preparation.

\bibitem[GS99]{Kirbycal1999}
Robert~E. Gompf and Andr\'as~I. Stipsicz.
\newblock {\em {$4$}-manifolds and {K}irby calculus}, volume~20 of {\em Graduate Studies in Mathematics}.
\newblock American Mathematical Society, Providence, RI, 1999.

\bibitem[Han23]{Hanselman2023cosmetic}
Jonathan Hanselman.
\newblock Heegaard {F}loer homology and cosmetic surgeries in {$S^3$}.
\newblock {\em J. Eur. Math. Soc. (JEMS)}, 25(5):1627--1669, 2023.

\bibitem[Hat07]{hatcher2007notes}
Allen Hatcher.
\newblock Notes on basic 3-manifold topology, 2007.

\bibitem[KM95]{KM1995structure}
P.~B. Kronheimer and T.~S. Mrowka.
\newblock Embedded surfaces and the structure of {D}onaldson's polynomial invariants.
\newblock {\em J. Differential Geom.}, 41(3):573--734, 1995.

\bibitem[KM04]{kronheimer04su2}
Peter~B. Kronheimer and Tomasz~S. Mrowka.
\newblock {Dehn surgery, the fundamental group and {SU}(2)}.
\newblock {\em Math. Res. Lett.}, 11(5-6):741--754, 2004.

\bibitem[KM07]{kronheimer2007monopoles}
Peter~B. Kronheimer and Tomasz~S. Mrowka.
\newblock {\em Monopoles and three-manifolds}, volume~10 of {\em New Mathematical Monographs}.
\newblock Cambridge University Press, Cambridge, 2007.

\bibitem[KM10]{kronheimer2010knots}
Peter~B. Kronheimer and Tomasz~S. Mrowka.
\newblock Knots, sutures, and excision.
\newblock {\em J. Differ. Geom.}, 84(2):301--364, 2010.

\bibitem[KM11a]{kronheimer2011khovanov}
Peter~B. Kronheimer and Tomasz~S. Mrowka.
\newblock Khovanov homology is an unknot-detector.
\newblock {\em Publ. Math. Inst. Hautes \'{E}tudes Sci.}, 113:97--208, 2011.

\bibitem[KM11b]{kronheimer2011knot}
Peter~B. Kronheimer and Tomasz~S. Mrowka.
\newblock Knot homology groups from instantons.
\newblock {\em J. Topol.}, 4(4):835--918, 2011.

\bibitem[KM22]{KM2022relation}
Peter Kronheimer and Tomasz Mrowka.
\newblock Relations in singular instanton homology.
\newblock arXiv:2210.07059, v1, 2022.

\bibitem[KMOS07]{kronheimer2007monopolesandlens}
Peter~B. Kronheimer, Tomasz~S. Mrowka, Peter~S. Ozsv\'{a}th, and Zolt\'{a}n Szab\'{o}.
\newblock Monopoles and lens space surgeries.
\newblock {\em Ann. of Math. (2)}, 165(2):457--546, 2007.

\bibitem[Kro97]{Kronheimer1997obstruction}
P.~B. Kronheimer.
\newblock An obstruction to removing intersection points in immersed surfaces.
\newblock {\em Topology}, 36(4):931--962, 1997.

\bibitem[LY24]{LY2022integral1}
Zhenkun Li and Fan Ye.
\newblock {Knot surgery formulae for instanton Floer homology I: the main theorem}.
\newblock arXiv:2206.10077, v3, 2024.

\bibitem[LY25a]{LY2025torsion}
Zhenkun Li and Fan Ye.
\newblock 2-torsion in instanton {F}loer homology.
\newblock {\em Adv. Math.}, 472:Paper No. 110289, 55, 2025.

\bibitem[LY25b]{LY20255surgery}
Zhenkun Li and Fan Ye.
\newblock {Instanton 2-torsion and Dehn surgeries}.
\newblock arXiv:2508.03394, v1, 2025.

\bibitem[Mos71]{moser1971elementary}
Louise Moser.
\newblock Elementary surgery along a torus knot.
\newblock {\em Pacific Journal of Mathematics}, 38(3):737--745, 1971.

\bibitem[OS05]{ozsvath2005double}
Peter~S. Ozsv\'{a}th and Zolt\'{a}n Szab\'{o}.
\newblock On the {H}eegaard {F}loer homology of branched double-covers.
\newblock {\em Adv. Math.}, 194(1):1--33, 2005.

\bibitem[OS06]{OS2006smooth}
Peter Ozsv\'ath and Zolt\'an Szab\'o.
\newblock Holomorphic triangles and invariants for smooth four-manifolds.
\newblock {\em Adv. Math.}, 202(2):326--400, 2006.

\bibitem[OSS15]{OSS2015grid}
Peter~S. Ozsv\'{a}th, Andr\'{a}s~I. Stipsicz, and Zolt\'{a}n Szab\'{o}.
\newblock {\em Grid homology for knots and links}, volume 208 of {\em Mathematical Surveys and Monographs}.
\newblock American Mathematical Society, Providence, RI, 2015.

\bibitem[Ozs94]{Ozsvath1994blowup}
Peter~S. Ozsv\'ath.
\newblock Some blowup formulas for {${\rm SU}(2)$} {D}onaldson polynomials.
\newblock {\em J. Differential Geom.}, 40(2):411--447, 1994.

\bibitem[Sca15]{scaduto2015instantons}
Christopher Scaduto.
\newblock Instantons and odd {K}hovanov homology.
\newblock {\em J. Topol.}, 8(3):744--810, 2015.

\bibitem[SS18]{SS2018quasi}
Christopher Scaduto and Matthew Stoffregen.
\newblock Two-fold quasi-alternating links, {K}hovanov homology and instanton homology.
\newblock {\em Quantum Topol.}, 9(1):167--205, 2018.

\bibitem[SZ22a]{sivek2022cyclic}
Steven Sivek and Raphael Zentner.
\newblock {$SU(2)$}-cyclic surgeries and the pillowcase.
\newblock {\em J. Differential Geom.}, 121(1):101--185, 2022.

\bibitem[SZ22b]{SZ2022surgery}
Steven Sivek and Raphael Zentner.
\newblock Surgery obstructions and character varieties.
\newblock {\em Trans. Amer. Math. Soc.}, 375(5):3351--3380, 2022.

\end{thebibliography}

\end{document}